\documentclass[11pt,a4paper,reqno]{amsart}%article}%
\title{Dg enhanced orbit categories and applications}
\author{Li Fan, Bernhard Keller and Yu Qiu}

\address{Fl: Department of Mathematical Sciences, Tsinghua University, 100084 Beijing, China.}
\email{fan-l17@tsinghua.org.cn}

\address{B. Keller: Universit\'e Paris Cit\'e and Sorbonne Université, CNRS, IMJ-PRG, F-75013 Paris, France}
\email{bernhard.keller@imj-prg.fr}

\address{Qy: Yau Mathematical Sciences Center and Department of Mathematical Sciences, Tsinghua University, 100084 Beijing, China. \&  Beijing Institute of Mathematical Sciences and Applications, Yanqi Lake, Beijing, China}
\email{yu.qiu@bath.edu}

\usepackage{amsfonts}
\usepackage{amsmath,amssymb,amsthm,amsxtra}
\usepackage{float}
\usepackage[colorlinks, linkcolor=blue!72,anchorcolor=orange,
    citecolor=red,urlcolor=Emerald, bookmarksopen,bookmarksdepth=2]{hyperref}
\usepackage[usenames,dvipsnames]{xcolor}
\usepackage{enumitem}
\usepackage{geometry,array}
\usepackage{graphicx}
\usepackage{subfigure}
\usepackage{bookmark}
\usepackage{tikz}
\usepackage{verbatim}
\usepackage{mathtools}
\usepackage{url}

\usetikzlibrary{matrix,positioning,decorations.markings,arrows,decorations.pathmorphing,	
    backgrounds,fit,positioning,shapes.symbols,chains,shadings,fadings,calc}
\tikzset{->-/.style={decoration={  markings,  mark=at position #1 with
    {\arrow{>}}},postaction={decorate}}}
\tikzset{-<-/.style={decoration={  markings,  mark=at position #1 with
    {\arrow{<}}},postaction={decorate}}}
\usetikzlibrary{cd}
\usepackage{cleveref}
\usepackage{mathabx}

\theoremstyle{plain}
\newtheorem{theorem}{Theorem}[section]

\newtheorem{lemma}[theorem]{Lemma}
\newtheorem{corollary}[theorem]{Corollary}
\newtheorem{proposition}[theorem]{Proposition}

\newtheorem{thmdef}[theorem]{Theorem/Definition}

\theoremstyle{definition}
\newtheorem{definition}[theorem]{Definition}

\newtheorem{example}[theorem]{Example}

\newtheorem{remark}[theorem]{Remark}

\numberwithin{equation}{section}
\newtheorem{assumption}[theorem]{Assumption}

\def\hua{\mathcal}
\def\hh{\mathcal}

\def\<{\langle}
\def\>{\rangle}
\def\={\simeq}

\def\to{\rightarrow}
\def\iso{\xrightarrow{\;\sim\;}}
\def\Lten{\stackrel{L}{\otimes}}
\def\NN{\mathbb{N}}
\def\ZZ{\mathbb{Z}}

\def\XX{\mathbb{X}}
\def\S{\mathbb{S}}

\renewcommand{\k}{\mathbf{k}}
\def\C{\hh{C}}
\newcommand{\D}{\operatorname{\hh{D}}}

\def\dim{\operatorname{dim}}
\def\rad{\operatorname{rad}}
\def\gldim{\operatorname{gldim}}

\def\REnd{\operatorname{REnd}}
\def\Fun{\operatorname{Fun}}
\def\dgFun{\Fun_{dg}}

\def\dg{\hh{A}} %dg category
\def\dgg{\hh{B}} %dg category
\def\dbg{\hh{B}} %dbg category
\def\dgcat{\operatorname{dgcat_{\k}}}
\def\pretr{\operatorname{pretr}}

\def\llEq{\operatorname{llEq}}
\def\ZZeq{\ZZ\text{-}\operatorname{Eq}}
\def\Zeq{\ZZ\text{-}\operatorname{Eq}(\dg,F,\dgg)}
\def\Nlleq{\mathbf{R}\NN\text{-}\llEq(\dg, F, \dgg)}
\def\CFZ{\dg/F^{\ZZ}}
\def\CllFN{\dg/_{ll}F^{\NN}}

\def\Cloc{(\dg/_{ll}F^{\NN})[q^{-1}]}

\def\i{i}
\def\p{p}
\def\qF{F_{\dg/\N}}
\def\N{\hh{N}}
\def\NFZ{\N/F_{\N}^{\ZZ}}
\def\repp{\operatorname{rep}}

\def\Hqe{\operatorname{Hqe}}
\def\hocolim{\operatorname{hocolim}}
\def\Xh{X^{\wedge}}
\def\Pairs{\operatorname{Pairs}}
\def\loll{\dg[S^{-1}]/_{ll}F'^{\NN}}
\def\loZ{\dg[S^{-1}]/F'^{\ZZ}}
\def\Grm{\operatorname{Grm}(\k)}
\def\Cdb{\C_{dbg}(\k)}
\def\kk{\mathbf{K}}
\def\sigm{\mathfrak{S}_m}
\def\AsZ{\hh{A}/\sigm^{\ZZ}}
\def\perz{\per^{\ZZ}}
\def\pvdz{\pvd^{\ZZ}}

\def\CFZb{\dg/^{dbg}F^{\ZZ}}
\def\CllFNb{\dg/^{dbg}_{ll}F^{\NN}}

\def\Cdgk{\hh{C}_{dg}(\k)}
\def\Cdgd{\hh{C}_{dg}(\dg)}
\def\QN{Q_{\NN}}
\def\QZ{Q_{\ZZ}}
\def\Ql{Q_{loc}}

\def\DXQ{\pvdz(\qq{\XX})}
\newcommand{\sslash}{\mathbin{/\mkern-6mu/}}
\def\DNQ{\pvd(\qq{N})}
\newcommand{\qq}[1]{\operatorname{\Pi}_{#1}A}

\def\obj{\operatorname{obj}}

\def\Hom{\operatorname{Hom}}
\def\Homz{\operatorname{Hom}^{\ZZ}}
\def\hom{{\hh{H}}om}

\def\rep{\operatorname{rep}_{dg}}
\def\Ho{\operatorname{Ho}}

\def\pvd{\operatorname{pvd}}
\def\sg{\operatorname{sg}}
\def\thick{\operatorname{thick}}

\renewcommand{\mod}{\operatorname{mod}}

\newcommand{\Cone}{\operatorname{Cone}}

\newcommand{\id}{\operatorname{id}}

\newcommand{\per}{\operatorname{per}}

\def\P{\mathbf{P}}
\def\RHom{\operatorname{RHom}}
\def\RHomz{\operatorname{RHom}^{\ZZ}}

\def\pixn{\pi_{N}}
\tikzcdset{arrow style=tikz, diagrams={>=stealth}}
\def\TEA{E}
\def\Ex{E_{\XX}}
\def\TAX{T}

\def\shiftX{\widehat{X}}

\newcommand\reserve[1]{}

\begin{document}
%=========================================================
\maketitle
%=========================================================
%=========================================================

%\tableofcontents\addtocontents{toc}{\setcounter{tocdepth}{1}}

\setlength\parindent{0em}
\setlength{\parskip}{5pt}
%=========================================================
\begin{abstract}
Our aim in this paper is to prove two results related to the three constructions
of cluster categories: as orbit categories, as singularity categories and as cosingularity categories.
In the first part of the paper, we prove the universal property of pretriangulated orbit categories of 
dg categories first stated by the second-named author in 2005. 
We deduce that the passage to an orbit category commutes 
with suitable dg quotients.
We apply these results to study collapsing of grading for  (higher) cluster categories 
constructed from bigraded Calabi--Yau completions as introduced by Ikeda--Qiu.

The second part of the paper is motivated by the construction of cluster categories
as (co)singularity categories.  We show that, for any dg algebra $A$,  its perfect derived category 
can be realized in two ways: firstly, as an (enlarged) cluster category of a certain
differential bigraded algebra, generalizing a result of Ikeda--Qiu, and secondly as a (shrunk) singularity 
category of another differential bigraded algebra, generalizing a result of Happel following Hanihara.
We relate these two descriptions using a version of relative Koszul duality.

\end{abstract}
\setcounter{tocdepth}{1}
\tableofcontents

%=========================================================
\section{Introduction}
%=========================================================
\paragraph{\textbf{Cluster categories and dg orbit categories}}
%=========================================================
Cluster algebras were introduced by Fomin and Zelevinsky \cite{FZ} in 2002 
and cluster phenomena have been spotted in various areas in mathematics (as well as in physics)
including algebraic/symplectic geometry, geometric topology and representation theory. 
A close link between cluster combinatorics and the representation theory of
finite-dimensional algebras was discovered by Marsh--Reineke--Zelevinsky
\cite{MRZ} in 2003. It was extended to a {\em categorification}
of the cluster algebra associated with an acyclic quiver $Q$ by Buan--Marsh--Reineke--Reiten--Todorov \cite{BMRRT},
who introduced the corresponding {\em cluster category $\C_2(\k Q)$} in 2006. Here $\k$ is a fixed field
and $\k Q$ the path algebra of $Q$. The cluster category is defined as the orbit category
\[
    \C_2(\k Q) = \D^b(\mod \k Q)/\S\circ[-2],
\]
where $\D^b(\mod \k Q)$ is the bounded derived category of the category of finitely generated
$\k Q$-modules, $\S$ its Serre functor and $[1]$ its shift functor.
In general, orbit categories of triangulated categories are not triangulated, but it was shown in \cite{K1} that 
$\C_2(\k Q)$ does carry a canonical triangulated structure. This structure is obtained in \cite{K1} using
dg (=differential graded) enhancements. Namely, it is proved that $\C_2(\k Q)$ is canonically equivalent to 
the $H^0$-category of the pretriangulated hull of the {\em dg orbit category} of the dg enhancement of the 
derived category $\D^b(\mod \k Q)$ under the action of the canonical dg lift of the triangle functor $\S \circ [-2]$. 
In section~9 of \cite{K1}, dg orbit categories of general pretriangulated dg categories are studied and
their universal property is stated in Theorems~9.3 and 9.6. It has been used for example in M.~Van den
Bergh's appendix to \cite{KellerReiten08}, in Amiot's fundamental paper \cite{A} and in Amiot--Oppermann's
\cite{AmiotOppermann14}. Our first aim in this article is to provide a careful statement and a detailed proof of
the universal property, cf.~Theorem~\ref{thm:Zeq}. 

\reserve{
The universal property of pretriangulated orbit categories of dg categories was first stated in \cite{K1} and later, Van den Bergh \cite{KellerReiten08} used the universal property to give a proof of the equivalence between some stable category of Frobenius category with good properties and the cluster category generated from its cluster tilting subcategory. Moreover, Amiot--Oppermann \cite{AmiotOppermann14} gave a more precise definition of the generalized cluster category via triangulated orbit category and realized the bounded derived category of a $G$-graded algebra as the triangulated orbit categories where $G$ is an Abelian group, by using the universal property.}

In more detail, for a dg category $\dg$ with a right perfect bimodule $F\in\rep(\dg,\dg)$, we first define the left lax quotient $\CllFN$ of $\dg$ (\Cref{def:CllFZ}) and show the universal property that it satisfies (\Cref{equi-CllNll}). We then define the dg orbit category $\CFZ$ as a dg localization of $\CllFN$ (\Cref{def:CFZ}). Notice that, even if $\dg$ is a pretriangulated, the dg orbit category $\CFZ$ may not be pretriangulated. A natural remedy is to take its pretriangulated hull $\pretr(\CFZ)$, which is unique in the Morita homotopy category of dg categories (\Cref{mainthm}). Therefore, we can naturally define a triangulated orbit category associated to some triangulated category $\hh{T}$ with a dg enhancement as follows:
\begin{thmdef}[\Cref{def:triorbit}]
Let $\hh{T}$ be a triangulated category endowed with a dg enhancement $H^0(\hua{A})\iso \hua{T}$ and let $F\in\rep(\dg,\dg)$ be a dg bimodule. If the induced functor $H^0(F): H^0(\hh{A})\to H^0(\hh{A})$ is an equivalence, then there is a canonical \emph{triangulated orbit category} of $\hh{T}$ under the action of $F$, which is defined as the category 
$H^0(\pretr(\hh{A}/F^{\ZZ}))$.
\end{thmdef}
The dg orbit category we have defined is functorial in a natural way. Moreover, taking suitable dg quotients commutes with taking dg orbit categories (see \Cref{cor:com}). The main application is the compatibility of the cluster category construction with
the collapsing of grading in Section~\ref{sec:5.2}, cf.~below.
As another application, in Section~\ref{ss:higher-cluster}, we extend the construction of the 
$m$-cluster category of \cite{Thomas07, BT} beyond the hereditary case.

The above universal characterization of dg orbit categories is related to Merlin Christ's recent preprint~\cite{C}, where he
studies group quotients in the $\infty$-categorical
setting. Here, for a (discrete) group $G$, a $G$-action on an $\infty$-category $\C$ is given by 
an $\infty$-functor from the classifiying space $BG$  to $\C$ and the group quotient 
is the homotopy colimit of this functor.
In Proposition~4.5 of \cite{C}, Christ shows that the dg orbit category under
the action of a Morita autoequivalence coincides with the $\infty$-categorical group quotient under
the associated $\ZZ$-action (cf.~Lemma~4.3 of \cite{C}). This confirms the fundamental
nature of the orbit category construction (and could be used to obtain a more sophisticated proof of
its universal property in the case of Morita autoequivalences).

\reserve{
%=========================================================
\paragraph{\textbf{The applications}}
%=========================================================
For a connective, smooth and proper dg algebra $A$ and an integer $m\geq 2$, since the perfect derived category $\per A$ has a canonical dg enhancement $\dg$, we can define its orbit $m$-cluster category from our construction of dg orbit categories (see \Cref{def:orbitmclu}). 
\begin{definition}[\Cref{def:orbitmclu}]
We let $\Sigma_m$ be the composition of the Serre functor $\S$ with $[-m]$. The \emph{orbit $m$-cluster category} associated to $A$ is defined as the triangulated orbit category $H^0(\pretr(\AsZ))$, where $\sigm$ is a dg replacement of $\Sigma_m$.
\end{definition}
In particular, when $A$ is hereditary, it coincides with the cluster category in the sense of \cite{BMRRT}. When $\gldim A\leq m$, it is triangle equivalent to the generalized $m$-cluster category.
}

%=========================================================
\paragraph{\textbf{Cluster categories via (co)singularity categories}}
%=========================================================
With motivations from (higher) cluster combinatories, Amiot--Guo--Keller \cite{A,G,K4} defined a generalized version of the (higher) cluster category as the {\em cosingularity category}
\[
\C_\Gamma = \per(\Gamma)/\pvd(\Gamma)
\]
of the (higher) Ginzburg dg algebra $\Gamma$, a certain smooth dg algebra invented by Ginzburg in \cite{Ginzburg06}.
Here $\per(\Gamma)$ denotes the perfect derived category and $\pvd(\Gamma)$ the perfectly valued derived
category, whose objects are the dg $\Gamma$-modules whose underlying complexes of vector spaces are perfect.
Motivated by questions about Fukaya categories, mirror symmetry and $q$-deforma\-tions of Bridgeland's stability conditions, Ikeda--Qiu \cite{IQ1, IQ2} further generalized Amiot--Guo--Keller's construction to the differential bigraded (dbg for short) case: 
For  the path algebra $A$ of an acyclic quiver, they defined \cite{IQ1} its \emph{$\XX$-Calabi--Yau completion} 
$\Pi_{\XX}A$ as a certain differential bigraded tensor algebra, cf.~\eqref{eq:pix}, and introduced
and studied $q$-stability conditions on its perfectly valued derived category $\pvdz(\Pi_{\XX}A)$.
As an interesting byproduct, they showed that the {\em $\infty$-cluster category of $A$},
defined as the cosingularity category
\[
\C^{\ZZ}(\Pi_{\XX}A)=\perz(\Pi_{\XX}A)/\pvdz(\Pi_{\XX}A),
\]
is equivalent to the perfect derived category $\per(A)$. 
Here, the Adams grading shift $[\XX]$ on $\C^{\ZZ}(\Pi_{\XX}A)$ corresponds
to the autoequivalence $\S\circ[1]$ on $\per A$.%, where $\S\coloneq?\Lten_ADA$.

Ikeda--Qiu's description of $\per(A)$ can be extended to arbitrary finite-dimensional
algebras of finite global dimension and is then reminiscent of Happel's description \cite{H}
of $\D^b(\mod A)$ as the {\em singularity category} (in the sense of Buchweitz \cite{B} and Orlov \cite{O})
of the trivial extension $\Ex= A\oplus(DA)[-\XX]$ of $A$. Our aim in Section~\ref{sec:RKD} is to
show how Ikeda--Qiu's and Happel's constructions are related in a very general setting:
Let $A$ be a dg algebra, considered as a dbg algebra concentrated in Adams degree 0, 
and let $X\in\D(A^e)$ be an invertible dg bimodule with inverse $Y$, cf. \eqref{eq:inverse}.
Define the \emph{enlarged cluster category} of $A$ with respect to $X$ as the Verdier quotient
\begin{equation}
\C^{\ZZ}(\TAX,A)=\perz(\TAX)/\pvdz(\TAX,A),
\end{equation}
where $\TAX=T_A(X[\XX-1])$ is the \emph{dbg tensor algebra}.
Dually, define the \emph{shrunk singularity category} of $A$ as the Verdier quotient
\begin{equation}
\sg^{\ZZ}(\TEA,A)=\pvdz(\TEA,A)/\perz(\TEA),
\end{equation}
where $\TEA$ is the \emph{dbg trivial extension algebra} $A\oplus Y[-\XX]$. We refer to \Cref{sec:RKD} for more details.
These dbg algebras and categories allow to realize $\per A$ in two ways,
which are related via a relative version of Koszul duality as follows.
\begin{theorem}[\Cref{thm:RKD}]\label{thm:RKDint}
We have the following commutative diagram
\begin{equation}\label{eq:comdiaint}
\begin{tikzcd}[column sep=30,row sep=12]
{\pvdz(\TAX,A)} \arrow[r, hook] \arrow[dd, "\wr"] & \perz(\TAX) \arrow[r] \arrow[dd, "{\wr. \RHomz_\TAX(A,?)}"] & {\C^{\ZZ}(\TAX,A)} \arrow[dd, "\wr"] & \\
  &  &  & \per A. \arrow[ld, "{\Psi,\sim}"] \arrow[lu, "{[1]\circ\Phi,\sim}"']\\
\perz(\TEA) \arrow[r, hook]   & {\pvdz(\TEA,A)} \arrow[r] & {\sg^{\ZZ}(\TEA,A)}  &
\end{tikzcd}
\end{equation}
\end{theorem}
For the compatibility of these functors with the Adams grading shifts, we refer to \Cref{thm:RKD}. 

In the final section~\ref{sec:5.2}, we consider, for a given $N\geq 3$, the functor which collapses
the components with bidegree $(a,b)$ on the diagonal $a+bN=p$ to the component of degree $p$.
We interpret it as the projection onto an orbit category and, as an application of the
compatibility between Verdier quotients and orbit quotients (\Cref{cor:com}), prove
in \Cref{thm:mainapp}, that the above diagram is compatible with the degree collapsing functor.

%=========================================================
\subsection{Contents}
%=========================================================
The article is organized as follows. In \Cref{sec:2}, we recall some basic concepts about dg/dbg categories. In \Cref{sec:3}, we give the construction of the dg orbit category explicitly and show the universal property that it satisfies (see \Cref{mainthm}). An application is defining the cluster categories for general dg algebras in \Cref{def:orbitmclu}. In \Cref{sec:RKD}, we realize the perfect derived categories as enlarged cluster categories and shrunk singularity categories respectively and connect them via a relative version of Koszul duality (see Theorems~\ref{thm:geniq}, \ref{thm:sg} and \ref{thm:RKD}). As an application of the commutation between orbit categories and
Verdier quotients, we obtain the exactness of the degree collapsing $2$-functor (see \Cref{thm:mainapp}).

%=========================================================
\subsection*{Acknowledgments}
%=========================================================
This article is part of L. Fan's Ph.~D. thesis under the supervision of Y. Qiu with an internship mentored by B. Keller. 
L. Fan is grateful to Y. Qiu for suggesting the problem and to B. Keller for his guidance and  patience. The authors thank Norihiro Hanihara for allowing them to include his Theorem~\ref{thm:sg}. B. Keller was supported by ANR project CHARMS (No.~ANR-19-CE40-0017). Y. Qiu is supported by the National Natural Science Foundation of China (Grants No.~12425104, 12031007 and 12271279) and the National Key R\&D Program of China (No.~2020YFA0713000). 
%=========================================================
%=========================================================
\section{Preliminaries}\label{sec:2}
%=========================================================
\subsection{Notations and Terminology}
%=========================================================
We fix a perfect field $\k$. All categories will be $\k$-categories.
The composition of morphisms is from right to left, that is, the composition of $f:X\to Y$ and $g:Y\to Z$ is $g\circ f:X\to Z$. If $\phi:K\to L$ is a morphism of complexes, we define its \emph{homotopy kernel} to be $\Sigma^{-1}\Cone(\phi)$.

Let $\hh{T}$ be a triangulated category and $\hh{P}$ a set in $\hh{T}$. We denote by $\thick(\hh{P})$ the smallest full triangulated subcategory of $\hh{T}$ containing $\hh{P}$ and closed under direct summands. Let $\hh{X}$ and $\hh{Y}$ be two full subcategories of $\hh{T}$. We denote by $\hh{X}*\hh{Y}$ the full subcategory of $\hh{T}$ consisting of objects $T\in\obj(\hh{T})$ such that there exists a triangle $X\to T\to Y\to X[1]$ with $X\in\obj(\hh{X})$ and $Y\in\obj(\hh{Y})$. Moreover, if we have $\Hom_{\hh{T}}(\hua{X},\hua{Y})=0$, we write $\hua{X}*\hua{Y}$ as $\hua{X}\bot\hua{Y}$. We write
\[\hh{X}{^\perp}=\{T\in\hh{T}\mid\Hom_{\hh{T}}(X,T)=0\}~~~\mbox{and}~{^\perp}\hh{X}=\{T\in\hh{T}\mid\Hom_{\hh{T}}(T,X)=0\}.\]

For a dg algebra $A$ over $\k$, we say $A$ is \emph{proper} if $A$ is perfect as a dg module over $\k$, $A$ is \emph{smooth} if $A$ is perfect as a dg module over $A^e = A^{op}\otimes_{\k}A$, and that $A$ is \emph{connective} if $H^i(A)=0$ for $i>0$.
%=========================================================
\subsection{Differential graded categories}
%=========================================================
A $\k$-category $\dg$ consists of the following data:
\begin{itemize}
\item a class of objects $\obj(\dg)$,
\item a $\k$-module $\dg(X,Y)$ for all $X, Y\in\obj(\dg)$, and
\item compositions which are $\k$-linear associative maps
\begin{align*}
\dg(Y, Z)\otimes_{\k}\dg(X,Y)&\to\dg(X, Z)\\g\otimes f&\mapsto g\circ f
\end{align*}
admitting units $\id_X\in\dg(X,X)$.
\end{itemize}
\begin{definition}\label{def:dgcat}
A \emph{differential graded category} (dg category for short) is a $\k$-category $\dg$ whose morphism spaces are complexes of $\k$-modules and whose compositions
\begin{align*}
\dg(Y, Z)\otimes_{\k}\dg(X,Y)&\to\dg(X, Z)\\g\otimes f&\mapsto g\circ f
\end{align*}
are chain maps, namely they satisfy the graded Leibniz rule \begin{equation}
d(g\circ f)=d(g)\circ f+(-1)^{|g|}g\circ d(f)
\end{equation} for all homogeneous morphisms $f$ and $g$.
\end{definition}
From the definition, we have that for each $X\in\dg,$ the identity $\id_X$ is a closed morphism of degree 0. The \emph{opposite dg category} $\dg^{op}$ of $\dg$ is the dg category with the same objects as $\dg$. The morphisms in $\dg^{op}$ are defined as $\dg^{op}(X, Y)= \dg(Y,X)$ and the composition of any homogeneous morphisms $f$ and $g$ is defined as $g\circ_{op} f= (-1)^{|f|\cdot|g|}f\circ g$. The \emph{underlying category} $Z^0(\dg)$ has the same objects as $\dg$ and morphisms which are defined by $$(Z^0\dg)(X, Y)= Z^0(\dg(X, Y)),$$ where $Z^0$ is the 0-th cocycle space of the complex $\dg(X,Y)$. The \emph{homotopy category} $H^0(\dg)$ has the same objects as $\dg$ and morphisms defined by $$(H^0\dg)(X, Y)= H^0(\dg(X, Y)),$$ where $H^0$ is the 0-th homology of $\dg(X,Y)$.
\begin{example}
We have the dg category $\Cdgk$ whose objects are all complexes $$V=(\bigoplus_{n\in \mathbb{Z}} V^n, d_V)$$ of $\k$-modules and whose morphism complexes ${\hom}(V, W)$ have their $p$th component given by
\begin{align*}
{\hom}(V, W)^p &=\prod_{n\in \mathbb{Z}} {\rm Hom}_{\k}(V^n, W^{n+p})\\
&=\{f\colon V\rightarrow W \; |\; f \mbox{ is homogeneous of degree } p\}.\end{align*}
The differential is given by \begin{equation}
d(f)=d_W\circ f-(-1)^{|f|} f\circ d_V.
\end{equation}
\end{example}
\begin{definition}\label{def:dgfun}
Let $\dg$ and $\dgg$ be two dg categories. A \emph{dg functor} $F:\dg\to\dgg$ is a $\k$-linear functor such that for any $X, Y\in\obj(\dg)$, the morphisms $$F(X,Y):\dg(X, Y)\to\dgg(FX, FY)$$ are chain maps.
\end{definition}
For a dg functor $F:\dg\to\dgg$, we define $\k$-linear functors $$Z^0(F): Z^0(\dg)\to Z^0(\dgg)~\mbox{and}~H^0(F): H^0(\dg)\to H^0(\dgg)$$ as follows.
We define $Z^0(F)(X)= F(X)$ and $Z^0(F)(X,Y)= Z^0F(X,Y)$, where $$Z^0F(X,Y):Z^0\dg(X,Y)\to Z^0\dgg(FX,FY)$$ is a $\k$-linear map for any objects $X, Y\in\obj(\dg)$. Analogously, We define $$H^0(F)(X)= F(X)~\mbox{and}~H^0(F)(X,Y)= H^0\dg(X,Y),$$ where we have $$H^0F(X,Y):H^0\dg(X,Y)\to H^0\dgg(FX,FY).$$

Let $F,G:\dg\to\dgg$ be two dg functors. The \emph{morphism complex} ${\hom}(F, G)$ has $p$th component given by data $\eta=(\eta_X)_{X\in {\obj}(\dg)}$ of morphisms $\eta_X \colon F(X)\to G(X)$ of degree $p$ in $\dgg$ satisfying\begin{equation}
G(f)\circ \eta_X =(-1)^{p\cdot |f|} \eta_{Y}\circ F(f)
\end{equation}
for all homogenous $f\colon X\to Y$. The differential is given by
\begin{equation}
(d\eta)_X=d_\dgg(\eta_X).
\end{equation}
Thus we have a dg category $\dgFun(\dg,\dgg)$, which is called \emph{functor dg category}, with objects as dg functors from $\dg$ to $\dgg$ and morphisms as above.
\begin{definition}
Let $\dg,\dgg$ be two dg categories. A \emph{quasi-equivalence} is a dg functor $F:\dg\to\dgg$ such that
\begin{itemize}
\item $F(X, Y):\dg(X, Y)\to\dgg(FX, FY)$ is a quasi-isomorphism for any objects $X, Y$ in $\dg$, and
\item the induced functor $H^0(F): H^0(\dg)\to H^0(\dgg)$ is essentially surjective.
\end{itemize}
We say that $F$ is \emph{quasi-fully faithful} if it satisfies the former condition, and \emph{quasi-essentially surjective} if it satisfies the latter one.
\end{definition}
The \emph{category of small dg categories} $\dgcat$ has all small dg categories as objects and dg functors as morphisms. It has a monoidal structure given by the tensor product of categories: Let $\dg$ and $\dgg$ be two dg categories. The dg category $\dg\otimes\dgg$ has objects given by $\obj(\dg)\times\obj(\dgg)$ and morphism spaces given by
\[(\dg\otimes\dgg)((X,X'),(Y,Y'))= \dg(X,Y)\otimes\dgg(X',Y')\]
for objects $X,Y$ in $\dg$ and $X',Y'$ in $\dgg$.
\begin{proposition}
The pair $(\dgcat,\otimes)$ is a closed symmetric monoidal category with internal Hom given by
\[\hom(\dg,\dgg)=\dgFun(\dg,\dgg).\]
\end{proposition}
By \cite{T}, there is a cofibrantly generated model structure on $\dgcat$ whose weak equivalences are the quasi-equivalences. We denote the corresponding homotopy category by
\[
    \Hqe=\Ho(\dgcat).
\]
%=========================================================
\subsection{Dg functors as dg modules}\label{sec:2.2}
%=========================================================
Let $\dg$ be a small dg category. A \emph{right dg $\dg$-module} is a dg functor
\[\dg^{op}\to\Cdgk.\]
We denote the dg category of right dg $\dg$-modules by $\Cdgd$. In other words, it is the functor dg category
\[
    \Cdgd=\dgFun(\dg^{op},\Cdgk).
\]
The Yoneda embedding of $\dg$ is the functor
\[\dg\to\Cdgd,\]
which is defined on objects by $X\mapsto\Xh= \dg(?,X)$. We have that the \emph{category of dg modules} $\C(\dg)$ is $Z^0(\Cdgd)$.

The \emph{derived category} $\D(\dg)$ is the localization of $\C(\dg)$ with respect to the class of quasi-isomorphisms. The \emph{perfect
derived category} $\per\dg$ is the full subcategory of compact objects of $\D(\dg)$. The \emph{perfectly valued derived category} $\pvd\dg$ is the full subcategory of dg modules $M$ of $\D(\dg)$ such that the underlying complex $M(X)$ is perfect over $\k$ for each $X$ in $\dg$. An important special case is the one when the dg category $\dg$ has a single object $\star$. We then identify $\dg$ with the dg algebra $A=\dg(\star,\star)$ and write $\D(A), \per A, \pvd A$ for the corresponding derived categories.

For two dg categories $\dg, \dgg$, a \emph{dg $\dg$-$\dgg$-bimodule} is a right $\dg^{op}\otimes\dgg$-module. We denote by $\rep(\dg,\dgg)$ the full subcategory of the dg derived category $\D_{dg}(\dg^{op}\otimes\dgg)$ whose objects are the cofibrant dg $\dg$-$\dgg$-bimodules $X$ which are \emph{right quasi-representable}, i.e. the $\dgg$-module $X(?,A)$ is quasi-isomorphic to a representable dg $\dgg$-module for each object $A$ of $\dg$. The details of dg derived categories can be found in \cite[Section 4.4]{K2}. We also set $\repp(\dg,\dgg)= H^0(\rep(\dg,\dgg))$. The composition
$$\rep(\dgg,\hh{E})\times\rep(\dg,\dgg)\to\rep(\dg,\hh{E})$$
is given by the tensor product $(Y,X)\mapsto YX= X\otimes_{\dgg} Y$. Notice that this coincides with the derived tensor product since we work over a field. Each $X\in\rep(\dg,\dgg)$ yields a derived functor
\[?\Lten_{\dg}X:\D(\dg)\to\D(\dgg),\]
and with the above definition, the derived functor associated with $X\otimes_{\dgg} Y$ is the composition of the functors
\[\begin{tikzcd}
\D(\dg)\arrow[r, "{?\Lten_{\dg}X}"]&\D(\dgg)\arrow[r, "{?\Lten_{\dgg}Y}"]&\D(\hh{E}).
\end{tikzcd}\]
We have a canonical dg functor
\begin{align}
\begin{split}\label{bijection}
\dgFun(\dg,\dgg)&\to\rep(\dg,\dgg)
\end{split}
\end{align}
taking a dg functor $F$ to $X_F:A\mapsto\D(?,FA)$.
By \cite{To}, we have a bijection between the set of isomorphism classes in $H^0(\rep(\dg,\dgg))$ and the set of morphisms $\dg\to\dgg$ in $\Hqe$.

%=========================================================
\subsection{Differential bigraded (dbg) categories}
%=========================================================
Let $\Grm$ be the category of $\ZZ$-graded $\k$-modules $M=\oplus_{p\in\ZZ}M_p$ with grading shift given by $(M[\XX])_p=M_{p+1}$.
It is a symmetric monoidal abelian $\k$-category.
Let $\Cdb$ be the dg category of complexes over $\Grm$, that is, we have $\Cdb=\C_{dg}(\Grm)$ and it is a monoidal dg category.
A \emph{differential bigraded category} (dbg category for short) $\dbg$ is a category enriched in $\Cdb$. In particular, if $\dbg$ is a dbg category with one object $\star$, then we can identify $\dbg$ with the dbg algebra $B=\dbg(\star,\star)$.

The \emph{bigraded derived category} $\D^{\ZZ}(\dbg)$ is the category of differential bigraded modules $M: \dbg^{op}\to\Cdb$ localized at the $s: M\to L$ such that $s_pX: L_pX\to M_pX$ is a quasi-isomorphism for any $p\in\ZZ$ and $X\in\obj(\dbg)$. It is not hard to check that, $\D^{\ZZ}(\dbg)$ is compactly generated by $\{\Xh[p\XX], p\in\ZZ, X\in\obj(\dbg)\}$. We define the perfect derived category $\per^{\ZZ}(\dbg)$ as the thick subcategory generated by $\{\Xh[p\XX], p\in\ZZ, X\in\obj(\dbg)\}$.  If $\k$ is regarded as a dbg algebra with trivial grading shift and we denote it $\kk$, then $\per^{\ZZ}(\kk)$ is the thick subcategory generated by $\kk[p\XX]$ for $p\in\ZZ$, i,e.
\[
    \per^{\ZZ}(\kk)=\{M\in\D(\kk)\mid\sum_{p,q\in\ZZ}\dim H^{p,q}(M)<\infty\}.
\]
Moreover, the perfectly valued derived category is
\[
    \pvd^{\ZZ}(\dbg)=\{M\in\D^{\ZZ}(\dbg)\mid MX\in\per^{\ZZ}(\kk), \forall X\in\obj(\dbg)\}.
\]
For $M,N\in\D^{\ZZ}(\dbg)$, we have that
\[\Hom_{\D^{\ZZ}(\dbg)}(L,M[p])=H^p\RHom_{\dbg}(L,M)\]
and
\[\Homz_{\D^{\ZZ}(\dbg)}(L,M[p])=H^p\RHomz_{\dbg}(L,M),\]
where $\Homz_{\D^{\ZZ}(\dbg)}(L,M[p])=\oplus_{q\in\ZZ}\Hom_{\D^{\ZZ}(\dbg)}(L,M[p+q\XX])$ and
\[\RHomz_{\dbg}(L,M[p])=\oplus_{q\in\ZZ}\RHom_{\dbg}(L,M[p+q\XX]).\]
\begin{remark}
If $A$ is a dg algebra considered as a dbg algebra concentrated in Adams degree 0, we have
\begin{equation}
\perz(A)=\thick(A[p\XX], p\in\ZZ)\subseteq\D^{\ZZ}(A).
\end{equation}
\end{remark}
\begin{remark}
The Adams $\ZZ$-grading can be any abelian group $(G,+)$ with identity $0$.
In this case, the dbg algebra is a dg $\k$-algebra endowed with a second $G$-grading, $B=\oplus_{g\in G}B_g$, such that the differential is of bidegree $|d|=(1,0)$, i.e. differential degree $1$ and Adams degree $0$.
A \emph{dbg $B$-module} is a dg $B$-module $M$ endowed with an Adams grading $M=\oplus_{g\in G}M_g$ such that the differential is of bidegree $|d_m|=(1,0)$.
There are two shift functors.
One is the shift of differential degree:
\[
    M\mapsto M[1]~\mbox{with}~(M[1])^p=M^{p+1}~\mbox{and}~d_{M[1]}=-d_M,
\]
and the other is the shift of Adams degree:
\[
    M\mapsto M[g]~\mbox{with}~(M[g])_h=M_{g+ h}~\mbox{and}~d_{M[g]}=d_M.
\]
\end{remark}

%=========================================================
\subsection{Dg localizations, dg quotients and pretriangulated dg categories}\label{sec:2.3}
%=========================================================
Let $\dg$ be a dg category and $S$ a set of morphisms of $H^0(\dg)$. We say that a morphism $G:\dg\to\dgg$ in $\Hqe$ makes $S$ \emph{invertible} if the induced functor $H^0(G):H^0(\dg)\to H^0(\dgg)$ takes each $s\in S$ to an isomorphism.
\begin{thmdef}[{\cite[Section 8.2]{To}}]
There is a universal morphism $\Ql:\dg\to\dg[S^{-1}]$ in $\Hqe$ which makes $S$ invertible, i.e. for any morphism $\dg\to\dgg$ of $\Hqe$ which makes $S$ invertible, there exists a unique morphism $\dg[S^{-1}]\to\dgg$ such that the following diagram
\begin{equation}
\begin{tikzcd}[column sep=32,row sep=30]
\dg \arrow[rd, "G"] \arrow[d, "\Ql"']           &    \\
{\dg[S^{-1}]} \arrow[r, "\exists ! \overline{G}"', dashed] & \dgg
\end{tikzcd}
\end{equation}
commutes. More precisely, the morphism $\Ql$ induces an isomorphism
\begin{equation}
    \Ql\colon\rep(\dg[S^{-1}],\dgg)\iso\rep(\dg,\dgg)_S
\end{equation}
in $\Hqe$, where $\rep(\dg,\dgg)_S$ is the full subcategory of $\rep(\dg,\dgg)$ with objects as dg bimodules that make $S$ invertible.
We call $\dg[S^{-1}]$ the \emph{dg localization} of $\dg$ at $S$.
\end{thmdef}
Let $\N$ be a full dg subcategory of $\dg$. We say that a morphism $G:\dg\to\dgg$ of $\Hqe$ \emph{annihilates} $\N$ if the induced functor $H^0(G):H^0(\dg)\to H^0(\dgg)$ takes all objects of $\N$ to zero objects.
\begin{thmdef}[{\cite{D,K2}}]
There is a universal morphism $\p:\dg\to\dg/\N$ in $\Hqe$ which annihilates $\N$, i.e. for any morphism $\dg\to\dgg$ of $\Hqe$ which annihilates $\N$, there exists a unique morphism $\dg/\N\to\dgg$ such that the following diagram
\begin{equation}
\begin{tikzcd}[column sep=30,row sep=30]
\dg \arrow[rd, "G"] \arrow[d, "\p"']           &    \\
{\dg/\N} \arrow[r, "\exists ! \overline{G}"', dashed] & \dgg
\end{tikzcd}
\end{equation}
commutes. More precisely, the morphism $\p$ induces an isomorphism
\begin{equation}
    \p\colon\rep(\dg/\N)\iso\rep(\dg,\dgg)_{\N}
\end{equation}
in $\Hqe$, where $\rep(\dg,\dgg)_{\N}$ is the full subcategory of $\rep(\dg,\dgg)$ with objects as dg modules that annihilate $\N$. We call $\dg/\N$ the \emph{dg quotient} of $\dg$ by $\N$.
\end{thmdef}

The small dg category $\dg$ is said to be \emph{pretriangulated} if the image of the Yoneda functor
\begin{align*}
H^0(\dg)&\to H^0(\Cdgd)
\end{align*}
is stable under shifts in both directions and extensions. We say that $\dg$ is \emph{strictly pretriangulated} if the image of the Yoneda functor
\begin{align*}
Z^0(\dg)&\to Z^0(\Cdgd)
\end{align*}
is stable under shifts in both directions in $Z^0(\Cdgd)$ and under graded split extensions. In this case $Z^0(\dg)$ is a Frobenius category when endowed with the graded split short exact sequences and $H^0(\dg)$ is a triangulated category.

\begin{thmdef}[{\cite[Section 4.5]{K2}}]
There is a canonical morphism in $\Hqe$ \begin{equation}
\dg\to\pretr(\dg),
\end{equation}
where $\pretr(\dg)$ is pretriangulated such that for any pretriangulated dg category $\dgg$ and any morphism $\dg\to\dgg$ of $\Hqe$, there exists a unique morphism $\pretr(\dg)\to\dgg$ which makes the following diagram
\begin{equation}\label{uni:pre}
\begin{tikzcd}[column sep=30,row sep=30]
\dg \arrow[rd, "G"] \arrow[d, hook] & \\\pretr(\dg) \arrow[r,"\exists ! \overline{G}"', dashed]  & \dgg
\end{tikzcd}
\end{equation}
commutative. We call $\pretr(\dg)$ the \emph{pretriangulated hull} of $\dg$ and define the \emph{triangulated hull} of $\dg$ as $H^0(\pretr(\dg))$.
\end{thmdef}
The functor $\pretr$ preserves quasi-equivalences \cite{D} and is left adjoint to the inclusion of the full subcategory of pretriangulated dg categories into the homotopy category $\Hqe$.

%=========================================================
\section{Construction and universal properties of orbit categories}\label{sec:3}
%=========================================================

%=========================================================
\subsection{The left-lax-equivariant category and quotient}\label{sec:llEq}
%=========================================================
Let $\dg_1,\dg_2,\dgg_1,\dgg_2$ be dg categories and $F_1\in\rep(\dg_1,\dg_2), F_2\in\rep(\dgg_1,\dgg_2)$ be two dg bimodules. We define a dg category associated to these data as follows.

\begin{definition}\label{pairsNlleq}
We define the dg category $\Pairs(F_1, F_2)$ as follows:
\begin{itemize}
\item objects are pairs $(G_1,G_2,\gamma)$, where $G_i$ is an $\dg_i$-$\dgg_i$-bimodule in $\rep(\dg_i,\dgg_i)$ for $i=1,2$ and $\gamma: G_2F_1\to F_2G_1$ between dg bimodules (which are not required to be isomorphisms); and
\item the complex of morphisms from $(G_1,G_2,\gamma)$ to $(G_1',G_2',\gamma')$ is defined as the homotopy kernel of the morphism
\begin{equation}
\Hom^{\bullet}_{\rep(\dg_1,\dgg_1)}(G_1,G_1')\oplus \Hom^{\bullet}_{\rep(\dg_2,\dgg_2)}(G_2,G_2')\to\Hom^{\bullet}_{\rep(\dg_1,\dgg_2)}(G_2F_1,F_2G_1')
\end{equation}
sending $(b_1,b_2)$ to $\gamma'\circ b_2F_1-F_2b_1\circ\gamma$.
\end{itemize}
\end{definition}
Concretely, a morphism of degree $n$ from $(G_1,G_2,\gamma)$ to $(G_1',G_2',\gamma'))$ is given by a pair $(b_1,b_2,h)$, where $b_i\in\Hom^n(G_i,G_i')$ for $i=1,2$ and $h\in\Hom^{n-1}(G_2F_1,F_2G_1')$ with differential given by
\begin{equation}\label{eq:d}
d(b_1,b_2,h)=(db_1,db_2, dh-(-1)^{|b_1|}F_2b_1\circ\gamma+(-1)^{|b_2|}\gamma'\circ b_2F_1).
\end{equation}
\[\begin{tikzcd}[column sep=30,row sep=30]
G_2F_1 \ar[r,"\gamma"]\ar[d,"b_2F_1"']\ar[rd,"h",dashed]&F_2G_1\ar[d,"F_2b_1"]\\ G_2'F_1\ar[r,"\gamma'"']&F_2G_1'\end{tikzcd}\]
The identity is the natural morphism $\id: G_i\to G_i$ with $h=0$ since the square above obviously commutes when $b_i=\id$ and $h=0$. The composition of $(b_1,b_2,h): (G_1,G_2,\gamma)\to(G_1',G_2',\gamma')$ with $(b_1',b_2',h'): (G_1',G_2',\gamma')\to(G_1'',G_2'',\gamma'')$ is given by
\begin{equation}
(b_1',b_2',h')\circ(b_1,b_2,h)=(b_1'b_1,b_2'b_2,F_2b_1'\circ h+(-1)^{|b_2|}h'\circ b_2F_1).
\end{equation}
\[\begin{tikzcd}[column sep=35,row sep=35]
G_2F_1 \ar[r,"\gamma"]\ar[d,"b_2F_1"]\ar[rd,"h",dashed]\arrow[dd, bend left=-30, "(b_2'\circ b_2)F_1"']&F_2G_1\ar[d,"F_2b_1"']\arrow[dd, bend left=30, "F_2(b_1'\circ b_1)"]\\ G_2'F_1\ar[r,"\gamma'"]\ar[d,"b_2'F_1"]\ar[rd,"h'",dashed]&F_2G_1'\ar[d,"F_2b_1'"']\\G_2''F_1\ar[r,"\gamma''"]&F_2G_1''
\end{tikzcd}\]
\begin{remark}
The category we defined above is a dg category.
Firstly, by the definition in \eqref{eq:d}, the map $d$ is a differential.
Moreover, one easily checks that the graded Leibniz rule holds.
\end{remark}
Let $\dg$ be a dg category and $F\in\rep(\dg,\dg)$ a dg bimodule.
\begin{definition}\label{Nlleq}
For a dg category $\dgg$, the \emph{$\NN$-left-lax-equivariant category} $\Nlleq$ of $\NN$-left-lax-equivariant functors from $(\dg,F)$ to $(\dgg,\id_{\dgg})$ is defined to be the dg category $\Pairs(F, \id_{\dgg})$.
\begin{remark}
Concretely, objects in $\Nlleq$ are pairs $(G, \gamma)$, where $G$ is an $\dg$-$\dgg$-bimodule in $\rep(\dg,\dgg)$  and $\gamma: GF\to G$ is a morphism of dg bimodules. Notice that we do not require $\gamma$ to be an isomorphism.
\end{remark}

For two dg categories $\dgg$ and $\dgg'$, a bimodule $H\in\rep(\dgg,\dgg')$ induces a functor from $\Nlleq$ to $\mathbf{R}\NN\text{-llEq}(\dg, F, \dgg')$, which sends $(G, \gamma)$ to $(HG, H\gamma)$.
\end{definition}
Next we define $\CllFN$, the functor $\QN: \dg\to \CllFN$ and the morphism of functors $q: \QN F\to \QN$ as follows.
\begin{definition}\label{def:CllFZ}
The dg category $\CllFN$, which is called the \emph{left lax quotient of $\dg$ by $F$,} is the dg category whose
\begin{itemize}
  \item objects are the same as the objects of $\dg$, and
  \item morphisms are given by $\CllFN(X, Y) = \bigoplus_{p\in \NN}\dg(X, F^pY)$.
\end{itemize}
\end{definition}
The identities and compositions are the natural ones:
\begin{itemize}
  \item identities: $\id_X\in\dg(X,X)\subseteq \bigoplus_{p\in \NN}\dg(X, F^pX)$;
  \item compositions: given $f\in \dg(X, F^pY)$ and $g\in \dg(Y, F^qZ)$, their composition is $g\diamond f=(F^pg)\circ f.$
\end{itemize}
\begin{remark}
Let us check that $\CllFN$ is indeed a dg category. The graded Leibniz rule\begin{equation}d(g\diamond f)= d(g)\diamond f+(-1)^{|g|}g\diamond d(f)\end{equation}
holds in $\CllFN$ because it holds in $\dg$ and $F$ is a dg functor. Indeed, we have
\begin{equation}
\begin{split}
  d(g\diamond f) & =d\big((F^pg)\circ f \big) = d(F^pg)\circ f+(-1)^{|F^pg|}(F^pg)\circ d(f)\\
   & = \big(F^pd(g)\big)\circ f+(-1)^{|g|}(F^pg)\circ d(f)\\
   &=d(g)\diamond f+(-1)^{|g|}g\diamond d(f).
\end{split}
\end{equation}
\end{remark}
\def\TCF{T_{\dg}(F)}
\begin{remark}\label{rmk:quasi}
For a dg category $\dg$ and a cofibrant dg $\dg$-$\dg$-bimodule $F$, the dg category $\CllFN$ is isomorphic to the tensor category $\TCF$ defined by
\[\TCF=\dg\oplus F\oplus (F\otimes_{\dg}F)\oplus\cdots\oplus F^{\otimes_{\dg}^n}\oplus\cdots.\]
Indeed, we can regard $\dg$ as the identity bimodule $(X, Y)\mapsto\dg(X,Y)$ and we have an isomorphism of bimodules
\[\dg(X,F^pY)=\dg(X,F\otimes_{\dg}\cdots\otimes_{\dg}FY)=(F\otimes_{\dg}\cdots\otimes_{\dg}F)(X,Y),\]
using the translation between functors and bimodules. For an $\CllFN$-$\dgg$-bimodule $G$, there is an exact sequence:
\[\begin{tikzcd}[column sep=15,row sep=23]
0 \arrow[r] & \TCF\otimes_{\dg} F\otimes_{\dg} G \arrow[r, "\phi"] & \TCF\otimes_{\dg} G \arrow[rr] \arrow[rd] & &G \arrow[r]&0.\\
            &   &                                     & \Cone(\phi) \arrow[ru, "quasi-iso."'] &
\end{tikzcd}\]
\end{remark}
The canonical dg functor $\QN:\dg\to \CllFN$ acts on
\begin{itemize}
  \item objects: it sends $X\in\obj(\dg)$ to $X\in\obj(\CllFN)$, and
  \item morphisms: it sends $f:X\to Y$ to $f \in\dg(X,Y)\subseteq \oplus_{p\in \NN}\dg(X, F^pY)$.
\end{itemize}
The canonical morphism of dg functors $q: \QN F\to \QN$ acts on the objects of $\dg$ as
\begin{equation}
    qX =\id_{FX}\in\dg(FX,FX)\subseteq \bigoplus_{p\in \NN}\dg(FX, F^{p+1}X).
\end{equation}
We use the following diagram
\[\begin{tikzcd}
\QN FX:FX \arrow[rd, "\id_{FX}"] & &&&  \\
\QN X:X&FX&F^2X& F^3X& \cdots
\end{tikzcd}\]
to illustrate the definition.

For morphisms in $\CllFN$, we have the factorization property. That is, each morphism $f:\QN X\to \QN Y$ uniquely decomposes into a finite sum $f=f_0+f_1+\cdots+f_N, N\in\NN$ and each $f_p$ uniquely factors as $f_p = q^pY\circ \QN g_p$ for some $g_p: X\to F^pY.$ Here we write $q^pX$ to denote the composition of morphisms
\begin{equation}\begin{tikzcd}[column sep=39,row sep=30]
\QN X &\QN FX \arrow[l, "qX"'] & \QN F^2X \arrow[l, "qFX"'] & \cdots \arrow[l, "qF^2X"'] & \QN F^pX .\arrow[l, "qF^{p-1}X"'] \arrow[llll, "\eqcolon\,q^pX", bend left=18]
\end{tikzcd}\end{equation}
In the following proposition, we use the factorization property to check the universal property of the left lax quotient.
\begin{proposition}\label{equi-CllNll}
For any dg category $\dgg$, we have an isomorphism in $\Hqe$
\begin{align}\label{eq:equi-CllNll}
\begin{split}
    \Phi: \rep(\CllFN, \dgg) &\to \Nlleq\\
     G &\mapsto (G\QN, Gq).
\end{split}
\end{align}
\end{proposition}
\begin{proof}
We prove that it is a quasi-equivalence in $\dgcat$, which induces an isomorphism in $\Hqe$. We need to check that $\Phi$ is quasi-fully faithful and quasi-essentially surjective.

For $G\in\rep(\CllFN,\dgg)$, we have a quasi-isomorphism
\begin{equation}\label{quasi}
\Cone(\TCF\otimes_{\dg}F\otimes_{\dg}G\to\TCF\otimes_{\dg}G)\to G
\end{equation}
by \Cref{rmk:quasi}.
The composition $G\QN$ is the restriction of $G$ from $\TCF^{op}\otimes\dgg$ to $\dg^{op}\otimes\dgg$ and $G\QN F$ is the restricted bimodule $G$ composed with $F$. Clearly, the restriction $G\QN$ is still right quasi-representable. Since $\TCF$ is cofibrant as a left dg $\dg$-module and $G$ is cofibrant over $\TCF^{op}\otimes\dgg$, the bimodule $G\QN$ is cofibrant over $\dg^{op}\otimes\dgg$. Thus, we have $G\QN\in\rep(\dg,\dgg)$.

Recall that $\rep(\TCF,\dgg)$ is a full subcategory of $\D_{dg}(\TCF^{op}\otimes\dgg)$. For any two objects $G,G'\in\rep(\CllFN, \dgg)$, we have
\begin{align}\label{quasi-ff}
\begin{split}
&\Hom^{\bullet}_{\rep(\TCF,\dgg)}(G,G')\\
\cong&\RHom_{\D_{dg}(\TCF^{op}\otimes\dgg)}(G,G')\\
\iso&\RHom_{\D_{dg}(\TCF^{op}\otimes\dgg)}(\Cone(\TCF\otimes_{\dg}F\otimes_{\dg}G\to\TCF\otimes_{\dg}G),G')\\
\cong&\Hom^{\bullet}_{\rep(\TCF,\dgg)}(\Cone(\TCF\otimes_{\dg}F\otimes_{\dg}G\to\TCF\otimes_{\dg}G),G')\\
\cong&\Sigma^{-1}\Cone\big(\Hom^{\bullet}_{\rep(\TCF,\dgg)}(\TCF\otimes_{\dg}G,G')\\
    &\hskip 2cm \to\Hom^{\bullet}_{\rep(\TCF,\dgg)}(\TCF\otimes_{\dg}F\otimes_{\dg}G,G')\big)\\
\cong&\Sigma^{-1}\Cone\big(\Hom^{\bullet}_{\rep(\dg,\dgg)}(G,G')\to\Hom^{\bullet}_{\rep(\dg,\dgg)}(F\otimes_{\dg}G,G')\big)\\
=&\Hom_{\Nlleq}(\Phi G,\Phi G').
\end{split}
\end{align}
Here, the second quasi-isomorphism follows from \eqref{quasi}, the third isomorphism follows from the fact that both the cone and $G'$ are cofibrant objects, the fourth isomorphism follows from the property of complexes when taking cones, the fifth isomorphism follows from restriction, and the last one is by \Cref{Nlleq}. Moreover, by \eqref{quasi-ff}, we also obtain that $\Phi$ is a dg functor, i.e. compatible with the differentials in the corresponding morphism spaces.

Next, we show the essential surjectivity of $H^0\Phi$. Given an object $(H,\eta)\in \Nlleq$, we construct the corresponding object $G\in \rep(\CllFN, \dgg)$ such that $G\QN=H$ and $Gq=\eta$ hold on objects. For any $\QN X\in\CllFN$ and $Z\in\dgg$, we define the right $\dgg$-module $G_0$ by
\begin{equation}G_0(Z,\QN X)= H(Z,X).\end{equation}
Note that $G_0$ is right quasi-representable because $H$ is. Now we define a left $\CllFN$-module structure on $G_0$. For any morphism $f: \QN X\to\QN Y$ in $\CllFN$, we define the left action of $f$ on $G_0$ using the decomposition $f=f_0+f_1+\cdots+f_N$ of $f$ as above. We define the left multiplication $\lambda_f:G_0(Z,\QN X)\to G_0(Z,\QN Y)$ by
\begin{equation}\lambda_f(g)=\sum_{p=0}^N \eta^p\circ\lambda_{f_p}(g),\end{equation}
where $g\in G_0(Z,\QN X)$ and $\lambda_{f_p}:H(Z,X)\to H(Z,F^pY)$ is given by the left $\dg$-module structure of $H$. Here we write $\eta^p$ for the composition of morphisms
\begin{equation}\begin{tikzcd}[column sep=39,row sep=30]
H & HF \arrow[l, "\eta "'] & HF^2 \arrow[l, "\eta F"'] & \cdots \arrow[l, "\eta F^2"'] & HF^p. \arrow[l, "\eta F^{p-1}"'] \arrow[llll, "\eqcolon\,\eta^p", bend left=18]
\end{tikzcd}\end{equation}
 We choose $G$ as a cofibrant resolution of $G_0$. Then $\Phi G$ is quasi-isomorphic to $(H,\eta)$.
Thus, $\Phi$ is a quasi-equivalence of dg categories.
\end{proof}
Let $S$ be a set of morphisms in $H^0(\dg)$ stable under $H^0(F)$. Denote by $\dg[S^{-1}]$ the dg localization of $\dg$ with respect to $S$, cf. \Cref{sec:2.3}, so that we have an isomorphism
\begin{equation}
\rep(\dg[S^{-1}],\dgg)\simeq\rep(\dg,\dgg)_S
\end{equation}
in $\Hqe$ for each dg category $\dgg$. In particular, it follows that $F$ induces a dg bimodule $F'\in\rep(\dg[S^{-1}],\dg[S^{-1}])$ such that the square
\begin{equation}
\begin{tikzcd}[column sep=30,row sep=30]
\dg \arrow[r,"F"] \arrow[d,"\Ql"'] &\dg\arrow[d,"\Ql"] \\\dg[S^{-1}] \arrow[r,"F'"]  & \dg[S^{-1}]
\end{tikzcd}
\end{equation}
commutes in $\Hqe$.

\begin{proposition}\label{prop:3.8dgloc}
Let $\dgg$ be a dg category. We have a canonical isomorphism in $\Hqe$
\begin{equation}
\rep(\loll,\dgg)\simeq\Nlleq_S,
\end{equation}
where the right hand side denotes the full subcategory of $\Nlleq$ whose objects are the pairs $(G,\gamma)$ such that $G$ makes $S$ invertible.
\end{proposition}
\begin{proof}
From \eqref{eq:equi-CllNll} in \Cref{equi-CllNll}, we have the isomorphism
\[\rep(\loll,\dgg)\simeq\mathbf{R}\NN\text{-}\llEq(\dg[S^{-1}],F',\dgg)\]
in $\Hqe$. On the other hand, the right hand side is exactly the dg category
\[\{(G,\gamma)\in\Nlleq\mid G~\mbox{makes}~S~\mbox{invertible}\},\]
which is $\Nlleq_S$.
\end{proof}

Let $\overline{S}$ be the image of $S$ under the projection functor $H^0(\dg)\to H^0(\CllFN)$.
\begin{corollary}\label{cor:lleqcomm}
There is a canonical isomorphism
\begin{equation}
(\CllFN)[\overline{S}^{-1}]\simeq\loll
\end{equation}
in $\Hqe$.
\end{corollary}
\begin{proof}
Let $\dgg$ be a dg category. We have canonical isomorphism in $\Hqe$
\begin{align}
\begin{split}
\rep(\loll,\dgg)&\simeq\Nlleq_S\\
&\simeq\rep(\CllFN,\dgg)_S\\
&=\rep(\CllFN,\dgg)_{\overline{S}}\\
&=\rep((\CllFN)[\overline{S}^{-1}],\dgg),
\end{split}
\end{align}
where the first equivalence is from \Cref{prop:3.8dgloc}, the second one is from \Cref{equi-CllNll} and the last one is by definition of dg localizations.
\end{proof}

%=========================================================
\subsection{The dg orbit category}\label{sec:dg-orbi}
%=========================================================

\begin{definition}\label{def:CFZ}
For a dg category $\dg$ with a bimodule $F\in\rep(\dg,\dg)$, the \emph{dg orbit category} $\CFZ$ is defined to be the dg localization $\Cloc$ of $\CllFN$ with respect to the morphisms $qX:\QN FX\to\QN X$ for any $X\in\obj(\dg)$.
\end{definition}

We denote the localization functor by $\Ql:\CllFN\to\CFZ,$ and define the quotient functor as the composition
\begin{equation}\label{eq:QZ}
    \QZ= \Ql\circ\QN: \dg\to\CFZ.
\end{equation}
For each object $X$ of $\dg$, we have a canonical quasi-isomorphism $$\Ql\circ qX:\QZ FX\to\QZ X.$$
\begin{remark}
Suppose that $F\in\rep(\dg,\dg)$ induces an equivalence $H^0F: H^0\dg\to H^0\dg$. Then for $X, Y\in\obj(\dg)$, the morphism space $\CFZ(X, Y)$ is quasi-isomorphic to the direct sum of the complexes $\dg(X, F^pY)$ for all $p\geq 0$ and $\dg(F^{-p}X, Y)$ for all $p<0$. Thus, we have an isomorphism
$$H^0(\CFZ)(X, Y)=\bigoplus_{p\in \ZZ}(H^0\dg)(X, (H^0F)^pY).$$
\end{remark}
\begin{definition}
We define the \emph{$\ZZ$-equivariant category} $\Zeq$ of $\ZZ$-equivariant functors from $(\dg,F)$ to $(\dgg,\id_{\dgg})$ as the full dg subcategory of $\Nlleq$ whose objects are the pairs $(G, \gamma)$, where $G$ is an $\dg$-$\dgg$-bimodule in $\rep(\dg,\dgg)$ and $\gamma: GF\to G$ is a morphism of dg bimodules such that $\gamma X$ is an isomorphism in $H^0(\dgg)$ for any $X\in\obj(\dg)$.
\end{definition}

\begin{theorem}\label{thm:Zeq}
Let $\dgg$ be a dg category. Then the dg functor $\QZ$ in \eqref{eq:QZ} induces
\begin{itemize}
\item an isomorphism in $\Hqe$:
\begin{equation}
\rep(\CFZ, \dgg) \iso \Zeq.
\end{equation}
\item a bijection from the set of morphisms $\CFZ\to\dgg$ in $\Hqe$ to the set of isomorphism classes in $H^0(\Zeq)$.
\end{itemize}
\end{theorem}
\def\repq{\operatorname{rep}_{dg,q}}
\begin{proof}
By combining the universal property of dg localization and \Cref{equi-CllNll}, we obtain the following isomorphisms in $\Hqe$
\begin{equation}\label{equi-Zeq}
\begin{array}{rcccl}
\rep(\CFZ, \dgg) &\iso& \repq(\CllFN,\dgg) &\iso& \Zeq\\
     G &\mapsto& G\circ\Ql &\mapsto& (G\QZ, G\Ql q).
\end{array}
\end{equation}
Here, $\repq(\CllFN,\dgg)$ denotes the full subcategory of $\rep(\CllFN,\dgg)$ consisting of bimodules whose associated functors $H^0(\CllFN)\to H^0(\dgg)$ send $qX$ to an isomorphism in $H^0(\dgg)$ for any $X\in\obj(\dg)$. This shows (1). Part (2) is an immediate consequence because isomorphism classes in $\repp(\CFZ,\dgg)=H^0(\rep(\CFZ,\dgg))$ are in bijection with morphisms $\CFZ\to\dgg$.
\end{proof}

Similar to (the proof of) \Cref{prop:3.8dgloc} and \Cref{cor:lleqcomm},
we have the following proposition and corollary.

\begin{proposition}
In the setting of \Cref{prop:3.8dgloc}, we have a canonical isomorphism in $\Hqe$
\begin{equation}
\rep(\loZ,\dgg)\simeq\Zeq_S,
\end{equation}
where the right hand side denotes the full subcategory of $\Zeq$ whose objects are the pairs $(G,\gamma)$, where $G$ makes $S$ invertible.
\end{proposition}
\begin{corollary}
We have a canonical isomorphism in $\Hqe$
\begin{equation}
    (\CFZ)[\overline{S}^{-1}]\simeq\loZ.
\end{equation}
\end{corollary}

Next, we assume that $\dg$ is a pretriangulated dg category. In general, even if $\dg$ is pretriangulated, the dg category $\CFZ$ is not necessarily pretriangulated. To get a pretriangulated quotient, we have to take the pretriangulated hull. For any pretriangulated dg category $\dgg$, we have an isomorphism
\begin{equation}
\rep(\pretr(\CFZ),\dgg) \iso\rep(\CFZ,\dgg)
\end{equation}
in $\Hqe$. By combining this with \Cref{thm:Zeq}, we obtain our main theorem below.

\begin{theorem}\label{mainthm}
Let $\dgg$ be a pretriangulated dg category. Then the dg functor \[\dg\to\pretr(\CFZ)\] induces
\begin{itemize}
\item an isomorphism in $\Hqe$
\begin{equation}
\rep(\pretr(\CFZ), \dgg) \iso \Zeq,
\end{equation}
and
\item a bijection from the set of morphisms $\pretr(\CFZ)\to\dgg$ in $\Hqe$ to the set of isomorphism classes in $H^0(\Zeq)$. In particular, for each pair $(G,\gamma)$ in $\Zeq$, there is a unique morphism $\overline{G}$ in $\Hqe$ making the triangle
\begin{equation}
\begin{tikzcd}[column sep=39,row sep=30]
    \dg \arrow[rd, "\forall G\in\Zeq"] \arrow[d] & \\
    \pretr(\CFZ) \arrow[r, dashed, "\overline{G}"']  & \dgg
\end{tikzcd}\end{equation}
commutative.
\end{itemize}
\end{theorem}
\begin{corollary}\label{cor:com}
Suppose $\N\subseteq\dg$ is a full dg subcategory such that $H^0(\N)$ is stable under $H^0(F)$, so that $F$ induces dg bimodules $F_{\N}\in\rep(\N,\N)$ and $\qF\in\rep(\dg/\N,\dg/\N)$. We have a canonical isomorphism in $\Hqe$
\begin{equation}\label{cor:commutes}
\pretr((\dg/\N)/\qF^{\ZZ})\simeq\pretr(\CFZ)/\pretr(\NFZ).
\end{equation}
\end{corollary}
\begin{proof}
For a pretriangulated dg category $\dgg$, we have canonical isomorphisms in $\Hqe$
\begin{align}
\begin{split}
\rep\big(\pretr((\dg/\N)/\qF^{\ZZ}),\dgg\big)&\simeq\ZZeq(\dg/\N,\qF,\dgg)\\
&\simeq\Zeq_{\N}\\
&\simeq\rep(\pretr(\CFZ),\dgg)_{\N}\\
&=\rep(\pretr(\CFZ),\dgg)_{\pretr(\NFZ)}\\
&\simeq\rep\big(\pretr(\CFZ)/\pretr(\NFZ),\dgg\big),
\end{split}
\end{align}
where $\Zeq_{\N}$ denotes the full subcategory of $\Zeq$ whose objects are the pairs $(G,\gamma)$ with $G$ annihilating $\N$.
Here, the first and third equivalences are from \Cref{mainthm}, the fourth equality is because that the image of $\N$ under the quotient $\dg\to\pretr(\CFZ)$ is $\pretr(\NFZ)$ and the second and the fifth equivalences are by the construction of dg quotients.
\end{proof}

In particular, if the pretriangulated dg category $\dg$ is a dg enhancement of some triangulated category $\hh{T}$, we can naturally define a triangulated orbit category associated to $\hh{T}$ as follows.
\begin{definition}\label{def:triorbit}
Let $\hh{T}$ be a triangulated category endowed with a dg enhancement $H^0(\hua{A})\iso \hua{T}$ and $F\in\rep(\dg,\dg)$ be a dg bimodule. If the induced functor $H^0(F): H^0(\hh{A})\to H^0(\hh{A})$ is an equivalence, then the \emph{triangulated orbit category} of $\hh{T}$ with respect to $\dg$ is defined as the category $H^0(\pretr(\hh{A}/F^{\ZZ}))$.
\end{definition}

%=========================================================
\subsection{Application: higher cluster categories without heredity assumption} \label{ss:higher-cluster}
%=========================================================
Let $A$ be a smooth, proper and connective dg algebra.
Our goal is to define its $m$-cluster category in a canonical way from \Cref{def:triorbit}, where $m\geq 2$ is an integer. 
The perfect derived category $\per A$ admits a Serre functor given by the left derived functor
\[\S= ?\Lten_ADA,\]
where $DA=\Hom_{\k}(A,\k)$ is the $\k$-dual of the bimodule $A$. We write the auto-equivalence
\[
    \Sigma_m=\S\circ[-m]
\]
of $\per A$. The triangulated category $\per A$ has a canonical dg enhancement given by the dg category $\hh{A}$ of perfect cofibrant dg $A$-modules. For any two objects $P,P'$ in $\dg$, the morphism space $\Hom_A^n(P, P')$ of degree $n$ consists of the homogeneous $A$-linear morphisms $f:P\to P'$ of graded objects of degree $n$. The differential is given by $$d(f)=d_{P'}\circ f-(-1)^nf\circ d_P.$$ Let $X\to DA$ be a cofibrant resolution of $DA$ in the category of dg $A$-bimodules. We take $\sigm= -\otimes_A X[-m]$ to be the tensor product
\begin{align}
\begin{split}
\sigm\colon\dg&\to\dg\\P&\mapsto P\otimes_A X[-m].
\end{split}
\end{align}
 We notice that $\sigm$ itself is not a dg equivalence but the induced triangle functor
\begin{equation}\begin{tikzcd}
H^0(\sigm):\arrow[d, "\simeq"',shift right=5]  H^0(\hh{A}) \arrow[r] \arrow[d, "\simeq",shift left=5]  & H^0(\hh{A}) \arrow[d, "\simeq"]\\ \Sigma_m:~~~ \per A \arrow[r,"\simeq"]&\per A
\end{tikzcd}
\end{equation}
is an equivalence. In general, the orbit category $\per A/\Sigma_m$ in the sense of \cite{K1} is not triangulated. However, we can use \Cref{def:triorbit} to obtain a triangulated category.
\begin{definition}\label{def:orbitmclu}
With the assumption above, the \emph{orbit $m$-cluster category} associated to $A$ is defined as the triangulated orbit category $H^0(\pretr(\AsZ))$.
\end{definition}
As shown in \cite{K1}, when $A$ is hereditary, the dg orbit category $\AsZ$ is already pretriangulated. So we have
\begin{equation}
H^0(\pretr(\AsZ))=H^0(\AsZ)\simeq H^0(\hh{A})/H^0(\sigm)\simeq\per A/\Sigma_m.
\end{equation}

%=========================================================
\subsection{Commutativity between orbit quotients and dg quotients}
%=========================================================
We continue to discuss the functoriality of the orbit category.
Let $\dg, \dg'$ be two dg categories with dg bimodules $F\in\rep(\dg,\dg)$ and $F'\in\rep(\dg',\dg')$. Let $G\in\rep(\dg,\dg')$ be a dg bimodule with an isomorphism $\gamma: GF\to F'G $ of $\rep(\dg,\dg')$. Let $\dgg$ be another dg category and $(H,\eta)$ be an object in $\ZZ$-Eq$(\dg',F',\dgg)$. Then the composition
\begin{equation}\begin{tikzcd}
HGF \arrow[r,"H\gamma"] & HF'G \arrow[r, "\eta G"]&HG
\end{tikzcd}\end{equation}
yields an object $(HG,\eta G\circ H\gamma)$ in $\Zeq$, which implies that $G$ induces a functor from $\ZZ$-$\text{Eq}(\dg',F',\dgg)$ to $\Zeq$. Thus we have an induced functor
\begin{equation}\overline{G}:\CFZ\to\dg'/F'^{\ZZ}\end{equation}
by \eqref{equi-Zeq}. Therefore, we deduce a functor between pretriangulated hulls and the corresponding triangulated categories:
\begin{equation}\pretr(\CFZ)\to\pretr(\dg'/F'^{\ZZ})\end{equation}
and
\begin{equation}H^0(\pretr(\CFZ))\to H^0(\pretr(\dg'/F'^{\ZZ})).\end{equation}
We make the following assumptions.
\begin{assumption}\label{assumption}
Let $\hh{T}$ be a triangulated category and $\hh{S}\subseteq\hh{T}$ a triangulated subcategory. We denote the corresponding Verdier quotient by $\hh{T}/\hh{S}$. Suppose we have a dg enhancement $\dg$ of $\hh{T}$ given by an equivalence $\Phi: H^0(\dg)\iso\hh{T}$. Let $F\in\rep(\dg,\dg)$ be a dg bimodule such that $\Phi\circ H^0(F)\circ\Phi^{-1}$ is an equivalence $\hh{T}\to\hh{T}$ and induces an equivalence $\hh{S}\to\hh{S}$. Let $\N\subseteq\dg$ be the full dg subcategory whose objects are those $X\in\obj(\dg)$ with $\Phi(X)$ in $\hh{S}$.
\end{assumption}
With the assumptions above, we have that $\N$ is a dg enhancement of $\hh{S}$ and, by the main result of \cite{D}, the dg quotient $\dg/\N$ becomes a dg enhancement for the Verdier quotient $\hh{T}/\hh{S}$. By the universal property of the dg quotient, the dg module $F$ then induces a bimodule $\qF\in\rep(\dg/\N,\dg/\N)$ such that $H^0(\qF)$ corresponds to the auto-equivalence of $\hh{T}/\hh{S}$ induced by $\Phi\circ H^0(F)\circ\Phi^{-1}$. We have the diagram
\[\begin{tikzcd}
\N \arrow[r, "\i"] \arrow["F_{\N}"', loop, distance=2em, in=125, out=55] & \dg \arrow[r, "\p"] \arrow["F"', loop, distance=2em, in=125, out=55] & \dg/\N, \arrow["\qF"', loop, distance=2em, in=125, out=55]
\end{tikzcd}\]
where $\i$ is inclusion and $\p$ the dg quotient functor. By the functoriality of the dg orbit categories, we obtain the induced sequence
\[\begin{tikzcd}
\N/F_{\N}^{\ZZ} \arrow[r, "\overline{\i}"] & \CFZ \arrow[r, "\overline{\p}"] & (\dg/\N)/\qF^{\ZZ}.
\end{tikzcd}\]
\begin{proposition}\label{cor:comm}
Under \Cref{assumption}, we have a short exact sequence of triangulated categories\begin{equation}\label{eq:exact}
0\to(\N/F_{\N}^{\ZZ})^{tr}\to(\CFZ)^{tr}\to((\dg/\N)/\qF^{\ZZ})^{tr}\to 0,
\end{equation}
where the notation $(-)^{tr}$ denotes the functor $H^0\circ\pretr$.
\end{proposition}
\begin{proof}
Since we work over a field, the conditions in \cite[Theorem 3.4]{D} hold naturally. We deduce from \Cref{cor:com} that there are equivalences
\[\bigg((\dg/\N)/\qF^{\ZZ}\bigg)^{tr}\simeq\bigg((\CFZ)/(\N/F^{\ZZ})\bigg)^{tr}\simeq\bigg(\CFZ\bigg)^{tr}/\bigg(\N/F^{\ZZ}\bigg)^{tr},\]
 which implies \eqref{eq:exact}.
\end{proof}
%=========================================================
\subsection{The bigraded case}
%=========================================================
Let $\dg$ be a dbg category and $F$ a dbg $\dg$-$\dg$-bimodule. Similar to before, we have the bigraded left lax quotient $\CllFNb$ of $\dg$ by $F$, which is a dbg category where
\begin{itemize}
\item the objects are the same as those of $\dg$, and
\item the morphisms are given by $\CllFNb(X, Y) = \bigoplus_{p \in \NN} \dg(X, F^pY)$.
\end{itemize}
By construction, $\CllFNb$ is equivalent to the tensor category $T_{\dg}(F)$. Moreover, we define the dbg orbit category $\CFZb$ as the dg localization of $\CllFNb$ with respect to the morphisms $qX: \QN FX \to \QN X$ for every $X \in \obj(\dg)$, where
\[\QN: \dg\to \CllFNb\]
is the quotient functor. The dbg orbit category $\CFZb$ also satisfies the universal property described in \Cref{thm:Zeq}. Furthermore, we can take its pretriangulated hull, thereby obtaining a triangulated structure. We have that all properties extend naturally to the bigraded setting. In particular, replacing the original condition in \Cref{assumption} with $F$ being a dbg $\dg$-$\dg$-bimodule, we obtain a short exact sequence of triangulated categories
categories\begin{equation}\label{eq:exactbi}
0\to(\N/^{dbg}F_{\N}^{\ZZ})^{tr}\to(\CFZb)^{tr}\to((\dg/\N)/^{dbg}\qF^{\ZZ})^{tr}\to 0.
\end{equation}
%=========================================================
\section{Cluster categories versus singularity categories}\label{sec:RKD}
%=========================================================
%=========================================================
\subsection{Relative perfectly valued derived categories} 
%=========================================================
Let $A$ be a dg algebra, which is a dbg algebra concentrated in Adams grading 0, $B$ a dbg algebra and $i_A^B:A \to B$ a morphism of dbg algebras. Such a morphism induces the restriction functors of corresponding derived categories
\[
    \D^{\ZZ}(B)\to\D^{\ZZ}(A).
\]

\begin{definition}\label{def:repvd}
We define the \emph{relative perfectly valued derived category} of $B$, with respect to $A$, to be
\begin{equation}
    \pvdz(B,A)=\{M\in\D^{\ZZ}(B)\big| M|_A\in\perz(A)\}.
\end{equation}
\end{definition}
\begin{lemma}
Suppose that $B=\bigoplus_{p\in\NN}B_p$ is positively Adams graded and $B_0=A$. Then $\pvdz(B,A)$ equals the thick subcategory of $\D^{\ZZ}(B)$ generated by the $A[q\XX], q\in\ZZ$. The same conclusion holds if $B=\bigoplus_{p\in\NN}B_{-p}$ is negatively Adams graded and $B_0=A$.
\end{lemma}
\begin{proof}
Clearly the thick subcategory generated by the $A[q\XX], q\in\ZZ$, is contained in $\pvdz(B, A)$. Conversely, if $M=\bigoplus_{q\in\ZZ}M_q$ is a dg $B$-module in $\pvdz(B,A)$, then the dg module $M_q$ over $B_0=A$ vanishes in $\D(A)$ for all $q\ll0$. If $M_q=0$ in $\D(A)$ for all $q\in\ZZ$, then $M$ is a zero object in $\D^{\ZZ}(B)$ and there is nothing to prove. So let us assume that $M_q\neq 0$ in $\D(A)$ for some $q\in\ZZ$. Let $q_0$ be the minimal such that $M_{q_0}\neq 0$ in $\D(A)$. Then we have an exact sequence of dbg $B$-modules
\begin{equation}
0\to M_{q_0}[-q_0\XX]\to M\to\oplus_{q>q_0}M_q\to 0.
\end{equation}
Since $M$ belongs to $\pvdz(B,A)$, the component $M_{q_0}$ belongs to $\per A$ and the right hand term still belongs to $\pvdz(B,A)$. By induction on the number of indices $q$ such that $M_q\neq 0$ in $\D(A)$, the right hand term belongs to the thick subcategory generated by the $A[q\XX],  q\in\ZZ$, and so does the left hand term. Thus, the dbg $B$-module $M$ belongs to this subcategory. The argument in the case when $B$ is negatively Adams graded is analogous (using the maximal index $q_0$ such that $M_{q_0}\neq 0$ in $\D(A)$).
\end{proof}

Now let $X\in\D(A^e)$ an invertible dg bimodule with inverse $Y$, that is, there are isomorphisms
\begin{equation}\label{eq:inverse}
X\Lten_A Y\simeq A~\mbox{and}~Y\Lten_A X\simeq A
\end{equation}
in $\D(A^e)$.
Write $\shiftX$ for a cofibrant resolution of $X[\XX-1]$. We let
\begin{equation}
\TAX=T_A(\shiftX)=\bigoplus_{p\geq 0}\shiftX ^{\otimes_A^p}
\end{equation}
be the \emph{differential bigraded tensor algebra of $\shiftX$ over $A$}
and
\begin{equation}
\TEA=A\oplus Y[-\XX]
\end{equation}
the \emph{differential bigraded trivial extension algebra of $Y[-\XX]$} over $A$. 

From the canonical inclusions $A\hookrightarrow\TAX$ and $A\hookrightarrow\TEA$,
we have the relative perfectly valued derived categories $\pvdz(\TAX,A)$ and $\pvdz(\TEA,A)$ respectively. Moreover, we have the canonical projections $\TAX\to A$ and $\TEA\to A$ so that $\TAX$ and $\TEA$ become augmented over $A$. Notice that $E$ is positively Adams graded and $\TAX$ is negatively Adams graded so that the above lemma applied to both $\pvdz(\TAX, A)$ and $\pvdz(\TEA, A)$.

%=========================================================
\subsection{Enlarged cluster categories}
%=========================================================
\begin{definition}
We define the \emph{enlarged cluster category} of $A$ with respect to $X$ as the Verdier quotient
\begin{equation}
\C^{\ZZ}(\TAX,A)=\perz(\TAX)/\pvdz(\TAX,A).
\end{equation}
\end{definition}
This generalizes Ikeda--Qiu's construction of the $\XX$-cluster category of a path algebra, cf. Section \ref{sec:5.2}. The inclusion (morphism of dg algebras) $A\to \TAX$ induces a triangle functor
\[
    \iota_A^\ZZ= ?\Lten_A \TAX:\per A\to\perz (\TAX)
\]
mapping $A$ to $\TAX$. Since that $X$ is cofibrant, the left derived tensor product becomes the usual tensor product.
\begin{theorem}\label{thm:geniq}
The composition
\begin{equation}
\begin{tikzpicture}[xscale=.9,yscale=.6]
\draw (-4,0)node{$\Phi:$};
\draw(180:3)node(o){$\per A$}(-3,2.2)node(b){\small{$?\Lten_AY[1]$}}
(-1.5,.5)node{$\iota_A^\ZZ$} (0,0)node(t){$\perz (\TAX)$};
\draw(0:3)node(a){$\C^{\ZZ}(T,A)$}(3,2.2)node(s){\small{$[\XX]$}};
\draw[->,>=stealth](o)to(t);\draw[->,>=stealth](b)to(s);\draw[->,>=stealth](t)to(a);
\draw[->,>=stealth](-3.2,.6).. controls +(135:2) and +(45:2) ..(-3+.2,.6);
\draw[->,>=stealth](3-.2,.6).. controls +(135:2) and +(45:2) ..(3+.2,.6);
\end{tikzpicture}
\end{equation}
is a triangle equivalence.
\end{theorem}
\begin{remark}\label{rmk:iq}
The theorem realizes the perfect derived category $\per A$ as the enlarged cluster category, which is a generalization of \cite[Theorem 6.7]{IQ1}.
\end{remark}
\begin{proof}
For $M\in\perz(\TAX)$, the $\TAX$-module structure $M\otimes_{A}\TAX\to M$ on $M$ is given by a dg $A$-module structure on $M$, together with a morphism of dg $A$-modules $M\otimes_A\shiftX\to M$, which is equivalent to
\[M\to M\otimes_{A}Y[1-\XX].\]
Taking the $q$-th components of the Adams grading on both sides for $q\in\ZZ$, we obtain
\[M_q\to M_{q-1}\otimes_A Y[1],\]
resulting in the following direct system
\[M_0\to M_{-1}\otimes_A Y[1]\to M_{-2}\otimes_A (Y[1])^{\otimes_A^2}\to\cdots\to M_{-q}\otimes_A (Y[1])^{\otimes^q_A}\to\cdots.\]
We now define the functor
\begin{align}
\begin{split}
\Psi:\perz(\TAX)&\to \per A\\
M&\mapsto \hocolim_q\big(M_{-q}\otimes_A (Y[1])^{\otimes^q_A}\big).
\end{split}
\end{align}
By construction, we have $\Psi(M[\XX])=M\otimes_A Y[1]$. Moreover, $\Psi$ vanishes on \[\{A[q\XX]\mid q\in\ZZ\},\] and hence also on $\pvdz(\TAX,A)$. Therefore, the functor $\Psi$ induces a functor from $\C^{\ZZ}(\TAX,A)$ to $\per A$.

For $L\in\per A$, we have
\begin{align}
\begin{split}
\Phi(L)&=L\otimes_A \TAX=L\otimes_A \bigoplus_{p\geq 0}\shiftX ^{\otimes_A^p}\\
&=L\oplus \big(L\otimes_A \shiftX\big)\oplus\big(L\otimes_A \shiftX ^{\otimes_A^2}\big)\oplus\cdots.
\end{split}
\end{align}
Thus, for each $q\in\ZZ$, we have
\begin{align}
\begin{split}
\Phi(L)_{-q}\otimes_A (Y[1])^{\otimes^q_A}&=L\otimes_A \shiftX ^{\otimes_A^q}[-q\XX]\otimes_A (Y[1])^{\otimes^q_A}\\
&=L\otimes_A (X\otimes_A Y)^{\otimes_A^q} \simeq L,
\end{split}
\end{align}
which implies that
\[\Psi\circ\Phi(L)=L.\]
In particular, we have
\[\Psi(\TAX)=\Psi\circ\Phi(A)=A.\]
For $q\in\ZZ$, we further compute 
\begin{align}\label{eq:com}
\begin{split}
\Phi\circ\Psi(\TAX[q\XX])&=\big(A\otimes_A (Y[1])^{\otimes_A^q}\big)\otimes_A \TAX\\
&=(Y[1])^{\otimes_A^q}\otimes_A\big(A\oplus \shiftX\oplus\cdots\oplus \shiftX ^{\otimes_A^p}\oplus\cdots\big)\\
&\simeq \bigoplus_{i=0}^{q-1} \big((Y[1])^{\otimes_A^{q-i}}[i\XX] \big)\oplus A[q\XX]\oplus \shiftX[q\XX]\oplus\cdots\oplus\shiftX ^{\otimes_A^p}[q\XX]\cdots\\
&=\bigoplus_{i=1}^{q-1} (Y[1])^{\otimes_A^{q-i}}[i\XX]\oplus\TAX[q\XX].
\end{split}
\end{align}
We define a natural transformation $\eta:\id_{\perz(\TAX)}\to\Phi\circ\Psi$ as follows. Since $\perz(\TAX)$ is the thick subcategory generated by $\{\TAX[q\XX], q\in\ZZ\}$, it suffices to define the evaluations of $\eta$ on the generators
\[\{\eta_{\TAX[q\XX]}:\TAX[q\XX]\to\Phi\circ\Psi(\TAX[q\XX])\}_{q\in\ZZ}\]
as the inclusion into the last component of \eqref{eq:com}, which is an isomorphism. Furthermore, $\eta:\id_{\C^{\ZZ}(\TAX,A)}\to\Phi\circ\Psi$ is also a natural isomorphism when restricted on $\C^{\ZZ}(\TAX,A)$, which implies the desired equivalence.

Finally, for $M\in\per A$, we have that
\[[\XX]\circ\iota_A^{\ZZ}(M)=(M\Lten_A T)[\XX]\]
and
\[\iota_A^{\ZZ}\circ(M\Lten_AY[1])=M\Lten_AY[1]\Lten_A T\simeq M\Lten_A(Y[1]\oplus T[\XX]),\]
which coincide in $\C^{\ZZ}(\TAX,A)$ because the object $M\Lten_A Y[1]$ belongs to $\thick(A[q\XX], q\in\ZZ)$ and hence vanishes in $\C^{\ZZ}(\TAX,A)$.
\end{proof}
\begin{remark}
In fact, we have 
\begin{equation}
\perz(\TAX)\simeq\big(\perz(A)/_{ll}^{dbg}(?\Lten_A\shiftX)^{\NN}\big)^{tr},
\end{equation}
which is the triangulated hull of the bigraded left lax quotient. When considering the Verdier quotient, we have that $?\Lten_A\shiftX$ becomes invertible in $\perz(\TAX)/\pvdz(\TAX,A)$, which corresponds to taking the dg localization $\perz(A)/^{dbg}(?\Lten_A\shiftX)^{\ZZ}$ with respect to $$M\Lten_A\shiftX\to M$$ for all $M\in\perz(A)$ in the bigraded left lax quotient. Therefore, we have the commutative diagram
\begin{equation}
\begin{tikzcd}
\perz(\TAX) \arrow[r, "\simeq"] \arrow[d, two heads] & \big(\perz(A)/_{ll}^{dbg}(?\Lten_A\shiftX)^{\NN}\big)^{tr} \arrow[d, two heads] \\
{\C^{\ZZ}(\TAX,A)} \arrow[r, "\simeq"]    &\big(\perz(A)/^{dbg}(?\Lten_A\shiftX)^{\ZZ}\big)^{tr},\end{tikzcd}
\end{equation}
where $\perz(A)/^{dbg}(?\Lten_A\shiftX)^{\ZZ}$ is automatically pretriangulated since it arises as a quotient of some perfect derived category. 
\end{remark}
%=========================================================
\subsection{Shrunk singularity categories}
%=========================================================
\begin{definition}
We define the \emph{shrunk singularity category} of $A$ as the Verdier quotient
\begin{equation}
\sg^{\ZZ}(\TEA,A)=\pvdz(\TEA,A)/\perz(\TEA).
\end{equation}
\end{definition}
The following theorem generalizes Happel's description \cite{H} of the derived category of a finite-dimensional algebra of finite global dimension. We learned about the theorem and its proof from Norihiro Hanihara \cite{Han}. It holds under the weaker assumption that the functor $\RHom_A(?,Y)$ induces an equivalence
\[
    \RHom_A(?,Y)\colon (\per A)^{op}\iso\per(A^{op}).
\]
Alternatively, the theorem follows from \Cref{thm:geniq} via the relative Koszul duality in \Cref{thm:RKD}.

\begin{theorem}[{Hanihara \cite{Han}}]\label{thm:sg}
The restriction along the augmentation $E\to A$ induces an equivalence
\begin{equation}
\begin{tikzpicture}[xscale=.6,yscale=.6]
\draw (-4.5,0)node{$\Psi:$};
\draw(180:3)node(o){$\per A$}(-3,2.2)node(b){\small{$?\Lten_AY[1]$}}
(0,.3)node{$\sim$};
\draw(0:3)node(a){$\sg^{\ZZ}(\TEA,A).$}(3,2.2)node(s){\small{$[\XX]$}};
\draw[->,>=stealth](o)to(a);\draw[->,>=stealth](b)to(s);
\draw[->,>=stealth](-3.2,.6).. controls +(135:2) and +(45:2) ..(-3+.2,.6);
\draw[->,>=stealth](3-.2,.6).. controls +(135:2) and +(45:2) ..(3+.2,.6);
\end{tikzpicture}
\end{equation}
\end{theorem}
\begin{remark}\label{rmk:sg}
The theorem generalizes the results of \cite[Theorem 10.10]{H} combined with \cite[Theorem 2.1]{R}.
\end{remark}
\begin{proof}
It is sufficient to show that $A$ generates $\sg^{\ZZ}(\TEA,A)$ and $\REnd_{\sg^{\ZZ}(\TEA,A)}(A)\simeq A$. For the first claim, we only need to show that $A[q\XX]\in\thick(A)$ for any $q\in\ZZ$. Since  there is a triangle
\[Y[-\XX]\to \TEA\to A\to Y[1-\XX]\]
in $\pvdz(E,A)$, we have
\[A\simeq Y[1-\XX]\]
in $\sg^{\ZZ}(\TEA,A)$. Thus we have $A[\XX]\in\thick(Y)=\thick(A)$, which implies that \[A[q\XX]\in\thick(A)\] for $q>0$. On the other hand, since $Y[-\XX]\in\thick(A)=\thick(Y)$, it follows that $A[q\XX]\simeq Y[(q-1)\XX+1]$ belongs to $\thick(Y)=\thick(A)$ for $q<0$.

By \cite[Proposition 1.9]{HI}, we have a triangle
\begin{equation}\label{eq:Htri}
A\Lten_{\TEA}\RHomz_{\TEA}(A, \TEA)\to\RHomz_{\TEA}(A, A)\to\RHomz_{\sg^{\ZZ}(\TEA,A)}(A,A)\to A\Lten_{\TEA}\RHomz_{\TEA}(A, \TEA)[1]
\end{equation}
in $\D^{\ZZ}(\k)$. Since we have $\RHom_A(A,Y)\cong Y$ and $\RHom_A(Y,Y)\cong A$, we obtain
\[\RHomz_A(\TEA,Y)\cong \TEA[\XX].\]
Thus, we have
\begin{align}\label{eq:adams1}
\begin{split}
A\Lten_{\TEA}\RHomz_{\TEA}(A, \TEA)&\cong A\Lten_{\TEA}\RHomz_{\TEA}(A, \RHomz_A(\TEA,Y)[-\XX])\\
&\cong A\Lten_{\TEA}\RHomz_{\TEA}(A\otimes_A\TEA, Y)[-\XX]\\
&\cong A\Lten_{\TEA}\RHomz_{\TEA}(\TEA, Y)[-\XX]\\
&\cong A\Lten_{\TEA}Y[-\XX].
\end{split}
\end{align}
The dbg algebra $\TEA$ is concentrated in Adams degrees $\geq 0$ and its component in Adams degrees 0 is $A$. It follows that if $M$ is a dg $\TEA$-module concentrated in Adams degrees non negative and whose underlying dg $A$-module is cofibrant, then $M$ has a cofibrant resolution $\widetilde{M}$ over $\TEA$ such that we have a short exact sequence
\[0\to M\otimes_A\TEA\to\widetilde{M}\to\widetilde{M}'\to 0\]
where $\widetilde{M}'$ is concentrated in Adams degrees $\geq 1$. If we apply this to $M=A$, we find that \eqref{eq:adams1} is concentrated in Adams degrees $\geq 1$ and that the component of $\RHom_{\TEA}(A,A)$ in Adams degree 0 is quasi-isomorphic to $\RHom_{A}(A,A)=A$. From the triangle \eqref{eq:Htri}, it follows that $\RHom_{\sg^{\ZZ}(\TEA,A)}(A,A)$ is quasi-isomorphic to $A$. Finally, the result of compatibility of $[X], ?\Lten_AY[1]$ and $\Psi$ follows from the one shown in \Cref{thm:geniq}, combined the relative Koszul duality in \Cref{thm:RKD}.
\end{proof}
%=========================================================
\subsection{Link via relative Koszul duality}
%=========================================================
\begin{theorem}\label{thm:RKD}
The adjoint pair
\begin{equation}
\begin{tikzcd}[column sep=60]
\perz(\TAX) \arrow[r, "{\RHomz_{\TAX}(A,?)}", shift left] & {\pvdz(\TEA,A)} \arrow[l, "?\Lten_{\TEA}A", shift left]
\end{tikzcd}
\end{equation}
induces the following commutative diagram
\begin{equation}\label{eq:comdia}
\begin{tikzcd}[column sep=36,row sep=20]
{\pvdz(\TAX,A)} \arrow[r, hook] \arrow[dd, "\wr"] & \perz(\TAX) \arrow[r] \arrow[dd, "{\wr. \RHomz_\TAX(A,?)}"] & {\C^{\ZZ}(\TAX,A)} \arrow[dd, "\wr"] & \\
  &  &  & \per A, \arrow[ld, "{\Psi,\sim}"] \arrow[lu, "{[1]\circ\Phi,\sim}"']\\
\perz(\TEA) \arrow[r, hook]   & {\pvdz(\TEA,A)} \arrow[r] & {\sg^{\ZZ}(\TEA,A)}  &
\end{tikzcd}
\end{equation}
where we have $[\XX]\circ[1]\circ\Phi\simeq[1]\circ\Phi\circ(?\Lten_AY[1]),[\XX]\circ\Psi\simeq\Psi\circ(?\Lten_AY[1])$ and all the other functors commute with $[\XX]$.
\end{theorem}
\begin{proof}
We have the equivalences $\Phi$ and $\Psi$ by \Cref{thm:geniq} and \Cref{thm:sg}. We now show the vertical equivalences. First, we show that $\RHomz_{\TAX}(A,A)=\TEA$, which implies the leftmost equivalence. We have the exact sequence of dbg $\TAX$-modules
\begin{equation}
0\to\shiftX\otimes_A\TAX\to\TAX\to A\to 0,
\end{equation}
where $\TAX$ and $\shiftX$ are assumed cofibrant as dg $A$-modules. It yields a cofibrant resolution (quasi-isomorphism) $\Cone(\shiftX\otimes_A\TAX\to\TAX)\to A$. We have quasi-isomorphisms
\begin{align*}
\begin{split}
\RHomz_{\TAX}(A,A)&\simeq\Homz_{\TAX}(\Cone(\shiftX\otimes_A\TAX\to\TAX), A)\\
&\simeq\Cone\big(\Homz_{\TAX}(T,A)\to\Homz_{\TAX}(\shiftX\otimes_A T,A)\big)[-1]\\
&\cong\Cone\big(A\to\Homz_{A}(\shiftX,A)\big)[-1]\\
&\simeq A\oplus \big(\RHomz_A(X[\XX-1],A)\big)[-1]\\
&\simeq A\oplus \big(\RHom_A(X,A)[1-\XX]\big)[-1]\\
&\simeq A\oplus Y[-\XX]=\TEA,
\end{split}
\end{align*}
where the last equivalence follows from the fact that $Y$ is inverse to $X$, so that
\[\RHom_A(X,A)\simeq Y.\]
So we have a quasi-isomorphism of differential bigraded $\k$-modules
\[\RHomz_{\TAX}(A,A)\simeq\TEA.\]
To show that it underlies an isomorphism in the homotopy category of dbg algebras, we endow the resolution $P=\Cone(\shiftX\otimes_A\TAX\to\TAX)$ with a structure of left dbg $\TEA$-module. By definition, we have $\TEA=A\oplus Y[-\XX]$. The action of $A\subseteq\TEA$ is given by the left dbg $A$-module structures on $\shiftX\otimes_A\TAX$ and $\TAX$. The left multiplication by $Y[-\XX]$ on $P=\TAX\oplus(\shiftX\otimes_A\TAX)[1]$ is given by the matrix
\[\begin{pmatrix}
  0 & \lambda \\
  0 & 0
\end{pmatrix},\]
where $\lambda:Y[-\XX]\otimes_A(\shiftX\otimes_A\TAX)[1]\to\TAX$ is induced by the evaluation $Y[-\XX]\otimes_A\shiftX[1]\to A$, which makes sense since we have
\[Y[-\XX]\otimes_A\shiftX[1]=Y\otimes_A X[-\XX+\XX-1+1]=Y\otimes_A X.\]
For the middle column, we use the same strategy as above to compute $\RHomz_{\TAX}(A,\TAX)$. We have quasi-isomorphisms
\begin{align}
\begin{split}
\RHomz_{\TAX}(A,\TAX)&\simeq\Homz_{\TAX}(\Cone(\shiftX\otimes_A\TAX\to\TAX), \TAX)\\
&\simeq\Cone\big(\Homz_{\TAX}(T,\TAX)\to\Homz_{\TAX}(\shiftX\otimes_A T,\TAX)\big)[-1]\\
&\cong\Cone\big(\TAX\to\Homz_{A}(\shiftX,\TAX)\big)[-1]\\
%&=\Cone\big(\TAX\to\Homz_{A}(\shiftX,A\oplus\shiftX\oplus\shiftX\otimes_A\shiftX\oplus\cdots)\big)[-1]\\
%&\cong\Cone\big(\TAX\to\Homz_{A}(\shiftX,A)\oplus\Homz_{A}(\shiftX,\shiftX)\oplus\Homz_{A}(\shiftX,\shiftX\otimes_A\shiftX)\oplus\cdots\big)[-1]\\
&\simeq\Cone\big(\TAX\to Y[1-\XX]\oplus\TAX\big)[-1]\\
&\simeq Y[-\XX]
\end{split}
\end{align}
of dbg $\k$-modules, where the fourth quasi-isomorphism follows since
\begin{align}
\begin{split}
\Homz_{A}(\shiftX,\TAX)&=\Homz_{A}(\shiftX,\bigoplus_{p\geq 0}\shiftX^{\otimes_A^p})\\
&=\bigoplus_{p\geq 0}\Homz_{A}(\shiftX,\shiftX^{\otimes_A^p})\\
&=\Homz_{A}(\shiftX,A)\oplus\bigoplus_{p\geq 1}\Homz_{A}(\shiftX,\shiftX^{\otimes_A^p})\\
&=Y[1-\XX]\oplus\TAX.
\end{split}
\end{align}
Since that $\thick(A[q\XX], q\in\ZZ)=\thick(Y[q\XX], q\in\ZZ)$, the middle equivalence follows.
Moreover, the third equivalence and the community of $\RHomz_T(A,?)$ and $[\XX]$ are directly from computation above.

For the rightmost triangle in \eqref{eq:comdia}, we consider the generator $A$ in $\per A$. We have that
\[\RHomz_{\TAX}(A, \iota_A^{\ZZ}(A))\simeq Y[-\XX],\]
which is equivalent to $A[-1]$ in $\sg^{\ZZ}(\TEA,A)$, hence we finish the proof.
\end{proof}
%=========================================================
\subsection{Examples}
%=========================================================
Let us consider two special situations.
%=========================================================
\paragraph{\textbf{Classical Koszul duality}}
%=========================================================
For the case when $A=\k, X=\k[1-\XX]$ and $Y=\k[\XX-1]$, we have
\[\TAX=T_A(\k)=\k[u],\]
where $u$ is of bidegree $(0,0)$ and
\[\TEA=\k\oplus \k[-1]=\k[v]/(v^2),\]
where $v$ is of bidegree $(0,0)$. Here we consider $\TAX$ and $\TEA$ as augmented dg $A$-algebras with vanishing differential. We see that $\TAX$ and $\TEA$ are Koszul dual to each other in the sense that
\[\RHom_{\TEA}(A,A)\simeq\TAX\]
and
\[\RHom_{\TAX}(A,A)\simeq\TEA.\]
%=========================================================
\paragraph{\textbf{The $\XX$-Calabi--Yau case}}
%=========================================================
If $A$ is a smooth, proper and connective dg algebra, then we can take
\[X=A^{\vee}~\mbox{and}~Y=DA,\]
where $A^{\vee}$ is a cofibrant replacement of the dg bimodule $\RHom_{A^e}(A,A^e)$ and $D$ denotes the $\k$-dual. In this case, the dbg algebras
\begin{equation}\label{eq:T}
\TAX=T_A(A^{\vee}[\XX-1])\eqcolon\qq{\XX}
\end{equation}
and
\begin{equation}\label{eq:E}
E=A\oplus(DA)[-\XX]\eqcolon\Ex
\end{equation}
are the $\XX$-Calabi--Yau completion and the trivial extension of $A$ respectively. Moreover, \Cref{thm:RKD} generalizes both Ikeda--Qiu's and Happel's descriptions of the perfect derived category, cf. \Cref{rmk:iq} and \Cref{rmk:sg}.
%=========================================================
%=========================================================
\subsection{Application: $N$-reductions as taking orbits}\label{sec:5.2}
%=========================================================
Let $A$ be a connective, smooth and proper dg algebra.
We aim to generalize Corollary 6.8 of \cite{IQ1} and its dual part from finite-dimensional path algebras to any such dg algebra $A$  (see \Cref{thm:mainapp}).

%=========================================================
\paragraph{\textbf{$\infty$-cluster category}}
%=========================================================
We consider the $\XX$-version of the generalized cluster category in the sense of \cite{A, G, K4}. Recall that the $\XX$-Calabi--Yau completion \cite{IQ1} of $A$ is the differential bigraded tensor algebra
\begin{equation}\label{eq:pix}
\Pi_{\XX}A=T_A(A^{\vee}[\XX-1])=\bigoplus_{p\geq 0}(A^{\vee}[\XX-1])^{\otimes_A^p},
\end{equation}
where $A^{\vee}$ is a cofibrant replacement of the dg bimodule $\RHom_{A^e}(A,A^e)$ and $[\XX-1]$ indicates the shift by bidegree $(-1,1)$. Let $\perz(\qq{\XX})$ be the perfect derived category and $\DXQ$ be the perfectly valued derived category, cf. \cref{sec:2.2}. The \emph{$\infty$-cluster category} $\C^{\ZZ}(\qq{\XX})$ is defined by the short exact sequence of triangulated categories:
\begin{equation}\label{sesX}
\begin{tikzcd}
0 \arrow[r] & \DXQ \arrow[r]& \perz(\qq{\XX}) \arrow[r] & \C^{\ZZ}(\qq{\XX}) \arrow[r] &0.
\end{tikzcd}
\end{equation}
Notice that $\DXQ$ is $\XX$-Calabi--Yau, i.e. it has a Serre functor given by $[\XX]$.
%=========================================================
\paragraph{\textbf{Generalized $(N-1)$-cluster category}}
%=========================================================
Let $N\geq 3$ be an integer. Recall that the $N$-Calabi--Yau completion \cite{K4} of $A$ is the derived dg tensor algebra
\begin{equation}\label{eq:piN}
\Pi_{N}A=T_A(A^{\vee}[N-1])=\bigoplus_{p\geq 0}(A^{\vee}[N-1])^{\otimes_A^p}.
\end{equation}
Let $\per(\qq{N})$ and $\DNQ$ be its perfect derived category and perfectly valued derived category respectively. The \emph{generalized $(N-1)$-cluster category} $\C(\qq{N})$ is defined by the short exact sequence of triangulated categories:
\begin{equation}\label{sesN}
\begin{tikzcd}
0 \arrow[r] & \DNQ \arrow[r]& \per(\qq{N}) \arrow[r] & \C(\qq{N}) \arrow[r] &0.
\end{tikzcd}
\end{equation}
Notice that $\DNQ$ is $N$-Calabi--Yau and the generalized $(N-1)$-cluster category $\C(\qq{N})$ is Hom-finite if we have
\[\begin{cases}
    H^p(\qq{N})=0,& \text{for $p>0$},\\
    \dim H^0(\qq{N})<\infty;
\end{cases}\]
or equivalently, we have
\[\begin{cases}
    H^{p+N-1}(A^{\vee})=0, & \text{for $p>0$},\\
    H^{p(N-1)}((A^{\vee})^{\otimes_A^p})=0,& \text{when $p\gg 0$}.
\end{cases}\]
For example, the conditions hold if $A$ is a finite-dimensional algebra concentrated in degree $0$ and $\gldim A<N$. If $\C(\qq{N})$ is Hom-finite, then it is $(N-1)$-Calabi--Yau, cf. \cite[Theorem 4.9]{G}.
%=========================================================
\paragraph{\textbf{$N$-reductions}}
%=========================================================
We have a canonical isomorphism of vector spaces
\[\bigoplus_{(a,b)\in\ZZ^2, a+bN=p} (\qq{\XX})^a_b\to(\qq{N})^p.\]
If we denote the dg algebra $A\oplus DA[-N]$ by $\TEA_N$, then, on the level of 
differential (bi)graded algebras, we have the canonical projections
\begin{equation}\label{eq:projection}
    \pixn\colon\qq{\XX}\to\qq{N}~\mbox{and}~\pixn\colon \Ex\to\TEA_N
\end{equation}
collapsing the double degree $(a,b)\in\ZZ\oplus\ZZ\XX$ into $a+bN\in\ZZ$. They induce functors
\[
    \pixn\colon\perz(\qq{\XX})\to\per(\qq{N})~\mbox{and}~\pixn\colon\pvdz(\Ex)\to\pvd(\TEA_N),
\]
which restrict to the responding subcategories
\[\pixn\colon\DXQ\to\DNQ~\mbox{and}~\pixn\colon\perz(\Ex)\to\per(\TEA_N).\]
Next we prove our main application.
\begin{theorem}\label{thm:mainapp}
We have the following commutative diagram between short exact sequences of triangulated categories:
\begin{equation}\label{eq:mainapp}
\begin{tikzpicture}[xscale=1.3,yscale=.9]
\begin{scope}[shift={(0,-1)}]
\draw[cyan] (0,0)node(a1){$\pvd(\qq{N})$} (3,0)node(a2){$\per(\qq{N})$} (6,0)node(a3){$\C(\qq{N})$};
\end{scope}
\draw (0,3)node(b1){$\pvdz(\qq{X})$}(3,3)node(b2){$\perz(\qq{\XX})$}(6,3)node(b3){$\C^{\ZZ}(\qq{X})$};
%\draw[red] (.5,2)node{\tiny{$\sslash[\XX-N]$}} (3.5,2)node{\tiny{$\sslash[\XX-N]$}} (6.5,2)node{\tiny{$\sslash[\XX-N]$}};
\draw[thick,{Hooks[right]}-stealth,cyan] (a1)edge(a2);
\draw[thick,{Hooks[right]}-stealth] (b1)edge(b2);
\draw[thick,->,>=stealth,cyan] (a2)edge(a3);
\draw[thick,->,>=stealth] (b2)edge(b3);
\draw[thick,->,>=stealth,red] (b1)edge node[font=\tiny,right]{$\sslash[\XX-N]$} (a1) (b2)edge node[font=\tiny,right]{$\sslash[\XX-N]$}(a2) (b3)edge node[font=\tiny,right]{$\sslash[\XX-N]$}(a3);
\begin{scope}[shift={(-1.5,-1.4)}]
\begin{scope}[shift={(0,-1)}]
\draw[cyan] (0,0)node(c1){$\per(\TEA_N)$} (3,0)node(c2){$\pvd(\TEA_N)$} (6,0)node(c3){$\sg(\TEA_N)$} (9,0)node(e){$\C_{N-1}(A).$};
\end{scope}
\draw (0,3)node(d1){$\perz(\Ex)$}(3,3)node(d2){$\pvdz(\Ex)$}(6,3)node(d3){$\sg^{\ZZ}(\Ex)$} (9,3)node(f){$\per A$};
%\draw[red](.5,2.2)node{\tiny{$\sslash[\XX-N]$}} (3.5,2.2)node{\tiny{$\sslash[\XX-N]$}} (6.5,2.2)node{\tiny{$\sslash[\XX-N]$}}(9.8,2.2)node{\tiny{$/\tau^{-1}\circ[2-N]$}};
\draw[thick,{Hooks[right]}-stealth,cyan] (c1)edge(c2);
\draw[thick,{Hooks[right]}-stealth] (d1)edge(d2);
\draw[thick,->,>=stealth,cyan] (c2)edge(c3);
\draw[thick,->,>=stealth] (d2)edge(d3);
\draw[thick,->,>=stealth,red] (d1)edge node[font=\tiny,right]{$\sslash[\XX-N]$}(c1) (d2)edge node[font=\tiny,right]{$\sslash[\XX-N]$}(c2) (d3)edge node[font=\tiny,right]{$\sslash[\XX-N]$}(c3) (f)edge node[font=\tiny,right]{$/\tau^{-1}\circ[2-N]$}(e);
\end{scope}
\draw[thick,->,>=stealth,cyan] (e)edge(c3) (e)edge(a3);
\draw[thick,->,>=stealth] (f)edge(d3) (f)edge(b3);
\draw[cyan] (-.5,-1.8)node{$\sim$}(2.5,-1.8)node{$\sim$}(5.5,-1.8)node{$\sim$} (7.1,-1.8)node{$\sim$}(6,-2.2)node{$\sim$};
\draw (-.5,2.2)node{$\sim$}(2.5,2.2)node{$\sim$}(5.5,2.2)node{$\sim$} (7.1,2.2)node{$\sim$} (5.8,1.8)node{$\sim$};
\foreach \j in {1,...,3}
{\draw[thick,->,>=stealth,cyan](a\j)edge(c\j); \draw[thick,->,>=stealth](b\j)edge(d\j);}
\end{tikzpicture}
\end{equation}
\end{theorem}
\begin{proof}
The triangulated categories in black all have natural dg enhancements that satisfy \Cref{assumption} with $\pixn$ the corresponding collapsing functor. By \Cref{thm:RKD}, the collapsing functor $\pixn$ translates to $\S\circ[1-N]=\tau^{-1}\circ[2-N]$ under the triangle equivalences $[1]\circ\Phi:\per A\to\C^{\ZZ}(\qq{\XX})$ and $\Psi:\per A\to\sg^{\ZZ}(\Ex)$. Due to the uniqueness of the triangulated orbit categories in (the bigraded version of) Theorem \ref{mainthm}, we have that the categories in blue are the triangulated orbit categories respectively. By combining this result with (the bigraded version of) \Cref{cor:comm}, we obtain the commutative diagram \eqref{eq:mainapp}.
\end{proof}
\begin{remark}
The back side of \eqref{eq:mainapp} generalizes Corollary 6.8 of \cite{IQ1}, where the case of a finite-dimensional path algebra $A$ was treated. It confirms Conjecture 8.2 of \cite{IQ2} for those graded gentle algebras which are smooth, proper and connective.
Moreover, it also implies that the orbit $(N-1)$-cluster category in \Cref{def:orbitmclu} is triangulated equivalent to the generalized $(N-1)$-cluster category $\C(\qq{N})$ and coincides with its triangulated hull $\sg(\TEA_N)$ given by Keller in \cite[Section 7.1]{K1}.
\end{remark}

\begin{remark} Let us point out that
the setting of the above theorem is similar to the following situation considered in \cite{CHQ}: 
Let $A(\S,n)$ be the relative graded Brauer graph algebra (RGB algebra) associated with an S-graph $\S$ and a compatible integer $n$.
It is a dg algebra given by an explicit graded quiver with relations and a differential, cf. [loc.~cit.]. The generalized relative Ginzburg dg algebra 
$G(\S, n)$ is quasi-isomorphic to the Koszul dual of $A(\S,n)$. Let $\widetilde{\S}$ be a branched cover of $\S$ such that the quotient group is a finite group $K$. Then $G(\S, n)$ is the group quotient of $G(\widetilde{\S}, n)$ under the induced action of $K$. 
We have the following commutative diagram where the vertical arrows are quotients by $K$ (i.e. orbit categories) 
and the horizontal arrows stand for  Koszul duality
\[\begin{tikzcd}
{A(\widetilde{\S}, n)} \arrow[d, two heads] \arrow[r] & {G(\widetilde{\S}, n)} \arrow[d, two heads] \arrow[l] \\
{A(\S,n)} \arrow[r]                                   & {G(\S,n).} \arrow[l]                                  
\end{tikzcd}\]
\end{remark}

\appendix
%=========================================================
\section{An alternative proof for \Cref{thm:geniq} in the $\XX$-Calabi–Yau case}\label{sec:X-cluster}
%=========================================================
Let $A$ be a connective, smooth and proper dg algebra. We aim to realize the perfect derived category $\per A$ as its $\infty$-cluster category in an alternative way, which is a generalization of \cite[Theorem 6.7]{IQ1}. Let $\qq{X}$ be the $\XX$-Calabi--Yau completion of $A$. The inclusion (morphism of dg algebras) $A\to \qq{\XX}$ induces a triangle functor
\[
    \iota_A^\XX\colon ?\Lten_A\qq{\XX}:\per A\to\perz(\qq{\XX})
\]
mapping $A$ to $\qq{\XX}$.
\begin{lemma}
The functor $\iota_A^\XX$ is fully faithful.
\end{lemma}
\begin{proof}
It follows from direct computation
\begin{align}
\begin{split}
\Hom_{\perz(\qq{\XX})}(\qq{\XX},(\qq{\XX})[k])&=\Hom_{A}(A,(\qq{\XX})[k])\\
&=\Hom_{A}(A,(\bigoplus_{p\geq 0}(A^{\vee}[\XX-1])^{\otimes_A^p})[k])\\
&=\bigoplus_{p\geq 0}H^{k-p+p\XX}((A^{\vee})^{\otimes_A^p})\\
&=H^k(A)=\Hom_{\per A}(A,A[k]).
\end{split}
\end{align}
\end{proof}
We obtain the following theorem.
\begin{theorem}\label{thm:iq}
The composition
\begin{equation}\label{eq:main}
\begin{tikzcd}
    i_{*}:\per A \arrow[r, "\iota_A^\XX"] & \perz(\qq{\XX}) \arrow[r] & \C^{\ZZ}(\qq{\XX})
\end{tikzcd}
\end{equation}
is a triangle equivalence.
\end{theorem}
\begin{proof}
Our strategy is based on Theorem 1.1 of \cite{IY}. Since $A$ is smooth, the bimodule $A^{\vee}[\XX-1]$ is perfect and the $\XX$-Calabi--Yau completion $\Pi_{\XX}A$ of $A$ is smooth. Since $A$ is connective, we may assume that the truncation $\tau_{\leq 0}A=A$ so that we have a canonical dg algebra morphism $A\to H^0A$. Let $\{e_i\mid i\in I\}$ be a complete set of pairwise orthogonal primitive idempotents of $H^0 A, P_i=(\Pi_{\XX}A) e_i$ and $S_i$ the head of the indecomposable projective $H^0 A$-module $H^0P_i$. We also view the $S_i$'s as simple dg $A$-modules via the restriction along $A\to H^0 A$. Using the restriction along the canonical projection $\qq{\XX}\to A$, we also view the $P_i$'s and $S_i$'s as dg $\qq{\XX}$-modules. We define $\hh{S}\subseteq\pvdz(\qq{\XX})$ as the thick subcategory generated by $\{S_i\mid i\in I\}$ and $\hh{P}\subseteq\perz(\Pi_{\XX}A)$ the thick subcategory generated by $\{P_i\mid i\in I\}$.

We define
\[
    \hua{X}=\thick(\hua{S}[\ZZ_{\ge 0}\XX])
        \quad\text{and}\quad \hua{Y}=\thick(\hua{S}[\ZZ_{< 0}\XX]),
\]
which are stable under differential shifts in both directions. We write $\hua{X}^{\bot}$ and $^{\bot}\hua{Y}$ for the left and right orthogonals of $\hh{X}$ and $\hh{Y}$ in the perfect derived category $\perz(\qq{\XX})$.
We claim that
\[
    \hh{X}{^\perp}\cap\;{^\perp}\hh{Y}[1]= \hua{P}.
\]

By the orthogonality relations between ``projectives" and ``simples", we have that $$\Hom_{\perz(\qq{\XX})}(P_i[a+b\XX],S_j)$$ vanishes unless $i=j$ and $a=b=0$. Now, let $Z\in\perz(\qq{\XX})$. By Lemma 2.14 in \cite{P}, $Z$ is quasi-isomorphic to a minimal perfect dg module, cf. \cite[Definition 2.13]{P}, with underlying graded module
\[Z=\bigoplus_{\lambda\in\Lambda}P_{i_{\lambda}}[a_{\lambda}+b_{\lambda}\XX]\]
and differential $d$ whose components lie in
\[\rad(\qq{\XX})=\rad(A)\oplus\bigoplus_{p>0}A^{\vee}[\XX-1]^{\otimes_A^p}.\]
Therefore, the dimension of
\[\Hom_{\perz(\qq{\XX})}(Z,S_j[r+s\XX])=\bigoplus_{\lambda\in\Lambda}\Hom_{\perz(\qq{\XX})}(P_{i_{\lambda}}[a_{\lambda}+b_{\lambda}\XX],S_j[r+s\XX])\]
equals the number of $\lambda\in\Lambda$ such that $i_{\lambda}=j, a_{\lambda}=r$ and $b_{\lambda}=s$. By choosing $r\in\ZZ$ and $s<0$, we see that \[^{\bot}\hua{Y}=\thick(\hua{P}[\ZZ_{\ge0}\XX]).\]
By $\XX$-Calabi--Yau duality, we have
\[\Hom_{\perz(\qq{\XX})}(S_j[r+s\XX],Z)=D\Hom_{\perz(\qq{\XX})}(Z,S_j[r+(s+1)\XX]).\]
By choosing $r\in\ZZ$ and $s\geq0$, we see that \[\hua{X}^{\bot}=\thick(\hua{P}[\ZZ_{\le0}\XX]).\]
By our construction, the subcategory $^{\bot}\hua{Y}[1]$ equals $^{\bot}\hua{Y}$, hence $\hh{X}{^\perp}\cap\;{^\perp}\hh{Y}[1]= \hua{P}.$
Therefore, we obtain the theorem by combining this with \Cref{prop:pvd}, \Cref{prop:per} and applying Theorem 1.1 in \cite{IY}.
\end{proof}
\begin{proposition}\label{prop:pvd}
The perfectly derived category $\pvdz(\qq{\XX})$ is the thick subcategory generated by $\{S_i[a+b\XX], i\in I~\mbox{and}~a,b\in\ZZ\}$. In particular, we have a stable t-structure
$$\pvdz(\Pi_{\XX}A) = \hua{X}\bot\hua{Y}.$$ 
\end{proposition}
\begin{proof}
For a dg $\qq{\XX}$-module $M$, we have a canonical triangle
\[\sigma_{\leq p}M\to M\to\sigma_{>p}M\to\sigma_{\leq p}M[1],\]
where
\begin{equation}\label{eq:Xtru}
\sigma_{\leq p}M=\bigoplus_{j\leq p}M^{*,j} \quad\text{and}\quad \sigma_{>p}M=\bigoplus_{j> p}M^{*,j}.
\end{equation}
If $M$ is in $\pvdz(\qq{\XX})$, then we get a finite filtration $\sigma_{\leq p}M\subseteq M$ with the subquotient $\sigma_{\leq p}M/\sigma_{\leq p-1}M$. Then $\sigma_{\leq p}M/\sigma_{\leq p-1}M$ is concentrated in Adams grading $p$ with finite total dimension, hence in $(\pvd A)[-p\XX]$. On the other hand, for $N\in\pvd A$, the ordinary truncation functor gives a finite filtration $\tau_{\leq q}N\subseteq N$ with the subquotient $(H^qN)[-q]$. Since each $H^qN$ is of finite dimension for all $q\in\ZZ$, it is in the thick subcategory generated by $\{S_i, i\in I\}$, which implies the result. In particular, if we take $p=0$ in \eqref{eq:Xtru}, the corresponding $\XX$-truncation gives a canonical stable t-structure $$\pvdz(\Pi_{\XX}A) = \hua{X}\bot\hua{Y}.$$ 
\end{proof}
\begin{proposition}\label{prop:per}
Both $(\hua{X},\hua{X}^{\bot})$ and $(^{\bot}\hua{Y},\hua{Y})$ are torsion pairs in $\perz(\Pi_{\XX}A).$
\end{proposition}
\begin{proof}
From the discussion above, it suffices to show that they generate $\perz(\Pi_{\XX}A).$ For $\perz(\qq{\XX}) = ^\bot\hua{Y}\bot \hua{Y},$ we claim that: $\hua{P}\subseteq\hua{P}[\XX]\bot \hua{S}.$ Let us abbreviate $A^{\vee}[\XX-1]$ by $\Theta$, so that $\qq{\XX}$ can be written as $T_A\Theta$. There is a canonical triangle
\begin{equation}\label{eq:sesp}
T_A\Theta\otimes_A\Theta\to T_A\Theta\to A\to(T_A\Theta\otimes_A\Theta)[1].
\end{equation}
Since $A$ is proper and connective, the dg module $A$ is in $\hua{S}$. Since we can write $T_A\Theta\otimes_A\Theta$ as $(\qq{\XX})\otimes_AA^{\vee}[\XX-1]$ and $A^{\vee}\in\per(A^e)$, the dg module $T_A\Theta\otimes_A\Theta$ is in $\thick(\qq{\XX}\otimes_{\k}A)[\XX]$, which is in $\hua{P}[\XX]$. By combining this with the triangle \eqref{eq:sesp}, we obtain the claim.

Similarly, we have $$\hua{P}[-m\XX] \subseteq \hua{P}[-m\XX+\XX]\bot \hua{S}[-m\XX]$$ for any positive integer $m$. By induction, we have $$\thick(\hua{P}[\ZZ_{<0}\XX]) \subseteq\hua{P}\bot \thick(\hua{S}[\ZZ_{<0}\XX]).$$ Therefore, we have $$\perz(\qq{\XX}) =\thick(\hua{P}[\ZZ_{\geq0}\XX])\bot\thick(\hua{S}[\ZZ_{<0}\XX]) = {^\perp}\hua{Y}\bot \hua{Y}.$$ As for $\perz(\qq{\XX}) = \hua{X}\bot \hua{X}^\bot,$ we also claim that $\hua{P}[\XX]\subseteq\hua{S}\bot\hua{P}$. For dg modules $L,M\in\pvd A$, since $A$ is smooth and proper, we have
\begin{equation}
\Hom_{\D A}(M\Lten_AA^{\vee},L)\simeq D\Hom_{\D A}(L,M)\simeq\Hom_{\D A}(M,L\Lten_ADA).
\end{equation}
This shows that the functor $?\Lten_AA^{\vee}$ is left adjoint to $?\Lten_ADA$. Since the latter functor is an equivalence, it follows that $?\Lten_AA^{\vee}$ and $?\Lten_ADA$ are quasi-inverse equivalences. Let us write $\Theta'=DA[1-\XX]$. If we apply $?\otimes_A\Theta'$ to the canonical triangle
\[A[-1]\to T_A\Theta\otimes_A\Theta\to T_A\Theta\to A,\]
then we get that
\begin{equation}\label{eq:sespx}
A[-1]\otimes_A\Theta'\to T_A\Theta\to T_A\Theta\otimes_A\Theta'\to A\otimes_A\Theta'.
\end{equation}
By a similar discussion as before, we have that the first term is in $\hua{S}[-\XX]$ and the third term is in $\hua{P}[-\XX]$, hence the claim holds.

Similarly, we have $$\hua{P}[m\XX] \subseteq \hua{S}[(m-1)\XX]\bot\hua{P}[(m-1)\XX]$$ for any positive integer $m$. By induction, we have $$\thick(\hua{P}[\ZZ_{>0}\XX]) \subseteq\thick(\hua{S}[\ZZ_{\geq0}\XX])\bot\hua{P}.$$ Therefore, we have $$\perz(\qq{\XX}) =\thick(\hua{S}[\ZZ_{\geq0}\XX])\bot \thick(\hua{P}[\ZZ_{\leq0}\XX]) = \hua{X}\bot \hua{X}^\bot.$$
\end{proof}
Notice that the category $\D^{\ZZ}(\qq{\XX})$ is enriched over the derived category of \emph{graded} vector spaces (and so are its full subcategories $\perz(\qq{\XX})$ and $\DXQ$), but the quotient $\C^{\ZZ}(\qq{\XX})$ is only enriched over the derived category of vector spaces. We can obtain the following corollary directly from \Cref{thm:iq}.
\begin{corollary}\label{cor:X-clu}
The bigraded cluster category $\C^{\ZZ}(\qq{\XX})$ is Hom-finite. Moreover, it is its Serre functor $(\XX-1)$ induced from $\perz(\qq{\XX})$ becomes the Serre functor $\tau^{-1}\circ[1]$ of $\per A$ under the equivalence $i_*$ in \eqref{eq:main}.
\end{corollary}
%=========================================================
%=========================================================

\end{document}